\documentclass[11pt]{article}
\usepackage[margin =1in]{geometry}
\usepackage{amssymb,amsthm,amsmath}
\usepackage{enumerate}
\usepackage{hyperref}
\usepackage{cite}
\usepackage[capitalize]{cleveref}
\usepackage{enumitem}
\usepackage{color}
\usepackage{cancel}
\usepackage{pgf}
\usepackage{svg}
\usepackage{pgfplots}
\usepackage{import}

\usepackage[utf8]{inputenc} 
\usepackage[T1]{fontenc}    
\usepackage{hyperref}       
\usepackage{url}            
\usepackage{booktabs}       
\usepackage{amsfonts}       
\usepackage{nicefrac}       
\usepackage{microtype}
\usepackage{multirow}
\usepackage{xfrac}
\usepackage{todonotes}
\usepackage{titlesec}
\usepackage{caption}
\usepackage{subcaption}
\usepackage{graphicx}
\usepackage{wrapfig,lipsum}
\def\siam{0}

\titleformat{\subsubsection}[runin]
{\normalfont\normalsize\bfseries}{\thesubsubsection}{1em}{}

\usepackage{graphicx}

\usepackage{appendix}
\numberwithin{equation}{section}

\newcommand{\cV}{\mathcal{V}}

\newcommand{\conv}{\mathrm{conv}}

\newcommand{\cU}{\mathcal{U}}

\newcommand{\inclu}[0] {\ar@{^{(}->}}

\newcommand{\dist}{{\rm dist}}

\newcommand{\RR}{\mathbb{R}}

\newcommand{\NN}{\mathbb{N}}

\newcommand{\eps}{\varepsilon}

\newcommand{\deltanc}{\Delta_f}

\newcommand{\argmin}{\operatornamewithlimits{argmin}}

\newcommand{\mini}{\operatornamewithlimits{minimize}}

\newtheorem{thm}{Theorem}[section]
\newtheorem{theorem}{Theorem}[section]

\newtheorem{lemma}[thm]{Lemma}

\newtheorem{assumption}{Assumption}

\crefname{claim}{claim}{claims}
\Crefname{claim}{Claim}{Claims}
\crefname{lem}{lemma}{lemmas}
\Crefname{lem}{Lemma}{Lemmas}
\crefname{algorithm}{algorithm}{algorithms}
\Crefname{algorithm}{Algorithm}{Algorithms}

\theoremstyle{remark}
\newtheorem{claim}{Claim}

\usepackage{mathtools}
\DeclarePairedDelimiter{\dotp}{\langle}{\rangle}
\usepackage[algo2e, boxruled]{algorithm2e}
\SetKwIF{If}{ElseIf}{Else}{If}{}{Else if}{Else}{}
\SetKwFor{Step}{Step $k$: $(k \geq 0)$}{}{}
\SetKwFor{ParFor}{ParallelFor}{}{}
\SetKwFor{For}{For}{}{}

\usepackage[T1]{fontenc}
\usepackage{lmodern}

\begin{document}

	\title{Optimal Convergence Rates for the Proximal Bundle Method}


	 \author{Mateo D\'iaz\thanks{Center
	 for Applied Mathematics, Cornell University Ithaca, NY 14850, USA;
	 \texttt{people.cam.cornell.edu/md825/}}\qquad Benjamin Grimmer\thanks{ORIE, Cornell University Ithaca, NY 14850, USA;
	 \texttt{people.orie.cornell.edu/bdg79/} \newline
 	 This material is based upon work supported by the National Science Foundation Graduate Research Fellowship under Grant No. DGE-1650441.}}

	\date{}
	\maketitle

	\begin{abstract}
              We study convergence rates of the classic proximal bundle method for a variety of nonsmooth convex optimization problems. We show that, without any modification, this algorithm adapts to converge faster in the presence of smoothness or a H\"older growth condition. Our analysis reveals that with a constant stepsize, the bundle method is adaptive, yet it exhibits suboptimal convergence rates. We overcome this shortcoming by proposing nonconstant stepsize schemes with optimal rates. These schemes use function information such as growth constants, which might be prohibitive in practice. We provide a parallelizable variant of the bundle method that can be applied without prior knowledge of function parameters while maintaining near-optimal rates. The practical impact of this scheme is limited since we incur a  (parallelizable) log factor in the complexity. These results improve on the scarce existing convergence rates and provide a unified analysis approach across problem settings and algorithmic details. Numerical experiments support our findings.
	\end{abstract}
        \section{Introduction}\label{sec:intro}
Convex optimization has played a fundamental role in recent developments in high-dimensional statistics, signal processing, and data science. Large-scale applications have motivated researchers to develop first-order methods with computationally simple iterations. Although impressive in scope,  these methods often require delicate parameter tuning involving geometrical information about the objective function. Thus, imposing an obstacle for practitioners that rarely have access to such information.

In this work, we develop efficiency guarantees for \emph{proximal bundle methods}, which date back to the 70s, that solve unconstrained convex minimization problems
\begin{equation}\label{eq:main}
\mini_{x \in \RR^d} f(x)
\end{equation}
where $f: \RR^d \rightarrow \RR$ is a proper closed convex function. Throughout we assume $f$ attains its minimum value $f^*=\inf f$ on some nonempty set $X^* = \{x \mid f(x) = f^*\}$. {Our core finding is that classic bundle methods, without any modification, are adaptive,} which means that they speed up in the presence of smoothness or error bounds, with little to no tuning.

Proximal bundle methods were originally introduced by~\cite{Lemarechal1975}, \cite{mifflin1977algorithm}, and~\cite{Wolfe1975}. They are conceptually similar to model-based methods~\cite{Davis_2019,Ochs2019,Drusvyatskiy2021}. That is, methods that update their iterates by applying a proximal step to an approximation of the function, known as the model $f_k$:
\begin{equation}\label{eq:subproblem1}
  x_{k+1}  \leftarrow \argmin_{x} f_k(x) + \frac{\rho_k}{2}\|x - x_k\|^2 \quad \text{with} \quad \rho_{k} > 0.\end{equation}
Unlike these schemes, bundle methods only update their iterates $x_k$ when the decrease in objective value is at least a fraction of the decrease that the model predicted. Moreover, bundle methods incorporate information from past iterations into their models, allowing $f_k$ to capture more than the just objective's geometry near $x_k$.

This seemingly subtle change has a rather surprising consequence: the iterates generated by a bundle method, with \emph{any} constant parameter configuration, converge to a minimizer of $f$; see~\cite[Thm. 4.9]{Kiwiel1985}, \cite[Thm. XV.3.2.4]{hiriart2013convex}, or~\cite[Thm. 7.16]{Ruszczynski2006} for different variations of this result. This stands in harsh contrast to other classic first-order algorithms; for example, gradient descent and its accelerated variants rely on selecting a stepsize inversely proportional to the level of smoothness. Similarly, subgradient methods rely on carefully controlled decreasing stepsize sequences. These simpler algorithms may fail to converge when the stepsizes are not carefully managed. Thus, providing a compelling reason to consider bundle methods.

Although bundle methods are known to converge under a number of assumptions \cite{Kiwiel1985-nonconvex,Mifflin1982,Sagastizabal2005,Apkarian2009,Warren2010,Lv2019,deOliveira2019,Monjezi2019a,Monjezi2019b, correa1993convergence} and have been successfully used in applications \cite{Sagastizabal2012,deOliveira2014}, nonasympotic guarantees have remained mostly evasive. The purpose of this paper is to close this gap. We study convergence rates for finding an $\eps$-minimizer, e.g., $f(x) - \inf f \leq \eps$, under a variety of different assumptions on $f$. We consider settings where the objective function is either $M$-Lipschitz continuous
\begin{equation} \label{eq:Lipschitz}
|f(x) - f(y)| \leq M\|x-y\| \qquad \text{for all }x, y \in \RR^d
\end{equation}
or differentiable with an $L$-Lipschitz gradient, often referred to as $L$-smoothness,
\begin{equation} \label{eq:Smooth}
\|\nabla f(x) - \nabla f(y)\| \leq L\|x-y\| \qquad \text{for all }x, y \in \RR^d\ .
\end{equation}
In either setting, we investigate the method's rate of convergence with and without the presence of H\"older growth
\begin{equation} \label{eq:holder-growth}
f(x)- f^* \geq \mu \cdot \dist(x, X^*)^p \qquad \text{for all }x \in \RR^d\ ,
\end{equation}
where $\dist(x, S) = \inf_{y \in S}\|x-y\|$. Particularly important cases are when $p=1$ and $p=2$ which correspond to sharp growth ($\mu$-SG)~\cite{Burke1993} and a generalization of strong convexity, known as quadratic growth ($\mu$-QG).

Historically, bundle methods were primarily conceived for nonsmooth optimization. The reason for this is simple: solving the proximal bundle subproblem \eqref{eq:subproblem1} is believed to be expensive, which is acceptable for nonsmooth problems because of their intrinsic higher complexity. Yet, it is not acceptable for smooth problems. Thus, studying rates for smooth setting might seem odd. Nonetheless, recently Nesterov and Florea~\cite{Nesterov2105} introduced an efficient routine to tackle the subproblems and use it to derive a model-based method with sound theoretical and practical performance. This motivates our pursuit for efficiency guarantees for smooth optimization.

\subsection{Contributions}
The main contribution of this work is to establish convergence rates under every realizable combination of continuity/smoothness~\eqref{eq:Lipschitz} or~\eqref{eq:Smooth} and growth assumptions~\eqref{eq:holder-growth}, see Table~\ref{tab:bundle-rates}. Full theorem statements are given in Section~\ref{sec:bundle} and apply for any H\"older growth exponent (rather than just the cases of $p=1$ and $p=2$ shown in the table). Our analysis technique is fairly general as we apply it to every combination of assumptions as well as different stepsize rules. We show rates for the constant stepsize rule $\rho_k = \rho > 0$, which tend to be suboptimal. Yet, they improve under additional assumptions such as quadratic growth or smoothness. Tuning the constant $\rho$ to depend on a target accuracy $\epsilon$ yields faster convergence rates. Further, we propose nonconstant stepsize rules $\rho_k$ with two clear advantages: they yield yet faster convergence and their convergence does not slow down after reaching the target accuracy.

\begin{table}[h]
	{\small
		\begin{center}
			\renewcommand{\arraystretch}{1.6}
			\begin{tabular}{|c c|c|c|c|}\hline
				\multicolumn{2}{|c|}{Assumptions} & Rate for generic $\rho$ & Rate for tuned $\rho$ & Rate for adaptive $\rho_k$ \\ \hline
				\parbox[t]{2mm}{\multirow{3}{*}{\rotatebox[origin=c]{90}{$M$-Lipschitz}}} & No Growth & $O\left(\dfrac{M^2\|x_0-x^*\|^4}{\rho\epsilon^3}\right)$ & $O\left(\dfrac{M^2\|x_0-x^*\|^2}{\epsilon^2}\right)$ & $O\left(\dfrac{M^2\|x_0-x^*\|^2}{\epsilon^2}\right)$ \\
				& $\mu$-QG & $O\left(\dfrac{M^2}{\min\{\mu,\rho\}\epsilon}\right)$ & $ O\left(\dfrac{M^2}{\mu\epsilon}\right)$ & $ O\left(\dfrac{M^2}{\mu\epsilon}\right)$\\
				& $\mu$-SG & $O\left(\dfrac{M^2}{\rho\epsilon} \right)$ & $O\left( \dfrac{M^2}{\mu^2}\sqrt{\dfrac{\deltanc}{\epsilon}} \right)$ & $O\left(\dfrac{M^2}{\mu^2} \log\left(\dfrac{\deltanc}{\epsilon}\right)\right)$ \\ \hline
				\parbox[t]{2mm}{\multirow{2}{*}{\rotatebox[origin=c]{90}{$L$-Smooth}}} & No Growth & $O\left(\dfrac{L^3\|x_0-x^*\|^2}{\rho^2\epsilon}\right)$ & $O\left(\dfrac{L\|x_0-x^*\|^2}{\epsilon}\right)$ & $O\left(\dfrac{L\|x_0-x^*\|^2}{\epsilon}\right)$\\
				& $\mu$-QG & $O\left( \dfrac{L^3}{\rho^2\mu}\log\left(\dfrac{\deltanc}{\epsilon}\right)\right)$ & $O\left(\dfrac{L}{\mu} \log\left(\dfrac{\deltanc}{\epsilon}\right)\right)$ & $O\left(\dfrac{L}{\mu} \log\left(\dfrac{\deltanc}{\epsilon}\right)\right)$\\ \hline
			\end{tabular}
		\end{center}
	}
 \caption{Convergence rates. We denote $\deltanc := f(x_0) - \inf f$. Here QG and SG denotes quadratic and sharp growth, respectively. The first column applies for any choice of the parameter $\rho$, showing progressively faster convergence as more structure is introduced. The second column shows the rate after optimizing the choice of $\rho$. The third column further improves these by allowing nonconstant stepsizes $\rho_k$, albeit it requires information about the function that might be unavailable.}  \label{tab:bundle-rates}
\end{table}

The existing convergence theory for the proximal bundle method applies to settings comparable to the first two rows of our table. Kiwiel~\cite{Kiwiel2000} derived a $O(\epsilon^{-3})$ convergence rate for Lipschitz problems, which agrees with our theory. Du and Ruszczynski~\cite{Du2017} and subsequently Liang and Monteiro~\cite{Liang2021} showed a $O(\log(1/\epsilon)/\epsilon)$ convergence rate for Lipschitz, strongly convex problems, which we improve on by removing the extra logarithmic term and thus achieve the optimal convergence rate for this setting of $O(1/\epsilon)$.
To our knowledge, the rest of our convergence results apply to wholly new settings for the proximal bundle method. In all of the $M$-Lipschitz settings considered, we show that using a nonconstant stepsize the bundle method attains the optimal nonsmooth convergence rate.
In the $L$-smooth settings considered, the bundle method converges at the same rate as gradient descent. Although, unlike gradient descent, our convergence theory applies to any configuration of its algorithmic parameters.

To complement these results we propose a simple parallelizable variant of the bundle method that avoids the reliance on tuning a stepsize or sequence of stepsizes based on potentially unrealistic knowledge of underlying problem constants. This parallel method too falls under the umbrella of our analysis. It attains near optimal nonsmooth convergence rates for Lipschitz problems with any level of H\"older growth, {at the cost of running a logarithmic number of instances of the bundle method in parallel. This additional cost may make the method impractical in settings when processor power is limited. Its primary value lies in showing a parameter-free, optimal proximal bundle method is theoretically possible.}

\subsection{Related work}
In 2000, Kiwiel~\cite{Kiwiel2000} gave the first convergence rate for the proximal bundle method, showing that an $\epsilon$-minimizer $x_k$ is found with $k\leq O\left(\frac{\|x_0-x^*\|^4}{\epsilon^3}\right).$ More recently, Du and Ruszczy\'nski~\cite{Du2017} gave the first analysis of bundle methods when applied to problems satisfying a quadratic growth bound. In this case, an $\epsilon$-minimizer is found within $O(\log(1/\epsilon)/\epsilon)$ iterations. Following this, Liang and Monteiro~\cite{Liang2021} showed a variant of the proximal bundle method with proper stepsize selection attains the optimal convergence rate for convex and strongly convex optimization, up to logarithmic terms. This work continues this direction establishing a wide range of guarantees for an unmodified proximal bundle method. Note that our work assumes the given problem is unconstrained (matching the model of~\cite{Du2017} but falling short of that of~~\cite{Kiwiel2000,Liang2021}). Generalizing our analysis beyond the unconstrained setting provides an interesting future direction.

Despite historically having weaker convergence rate guarantees than simple alternatives like the subgradient method, bundle methods have persisted as a method of choice for nonsmooth convex optimization. See~\cite{Frangioni2020,Lemarechal2001} as a survey of much of the bundle method literature. In practice, bundle methods have proven to be efficient methods for solving many nonsmooth problems (see~\cite{Sagastizabal2005,Sagastizabal2012,deOliveira2014} for further discussion). Extensions that apply to nonconvex problems have been considered in~\cite{Kiwiel1985-nonconvex,Mifflin1982,Apkarian2009,Warren2010,Lv2019,deOliveira2019,Monjezi2019a,Monjezi2019b} and as well as extensions to problems where only an inexact first-order oracle is available in~\cite{Hare2016,deOliveira2020,Lv2018}. Bundle-based methods that exploit the so called $\cV\cU$-structure of the objective function were develop in \cite{mifflin2005algorithm, mifflin2012science}; under suitable conditions these algorithms exhibit superlinear convergence.

Stronger convergence rates have been established for related level bundle methods~\cite{Lemarechal1995}, which share many core elements with proximal bundle methods. Variations of level bundle methods were studied in~\cite{Kiwiel1995, de2017target, Lan2015}. {The results of Lan~\cite{Lan2015} provide guarantees across a range of problem settings, while incorporating Nesterov-type acceleration at the cost of needing two oracle calls per iteration. For nonsmooth optimization, their level bundle method rates align with the optimal rates we derive for the proximal bundle method. For smooth optimization, they achieve the optimal accelerated rate whereas our analysis finds the proximal bundle method only converges as fast as gradient descent.}


A few parallel bundle methods have been proposed in the literature \cite{fischer2014parallel, kim2019asynchronous, iutzeler2020asynchronous}. The authors of \cite{kim2019asynchronous} propose a bundle-trust-region targeted to stochastic mixed-integer programming problems. While \cite{iutzeler2020asynchronous} introduces an asynchronous level bundle method that assumes that $f$ is a sum of convex functions and exploits this decomposition to distribute computation. To our knowledge,  \cite{fischer2014parallel} introduced the only parallel variant of the \emph{proximal} bundle method, which seeks to reduce the computational load by greedily picking subsets of coordinates and executing parallel instances of the bundle method restricted to each subset. In contrast, our parallel bundle method
  aims to bypass the need for tuning by running a small number of instances with different stepsizes.

\paragraph{Outline}
Section \ref{sec:bundle} introduces the Proximal Bundle Method and provides the formal convergence guarantees under different regularity assumptions. This section also introduces simple stepsize rules that guarantee optimal convergence rates for all nonsmooth settings. Practical implementations of these rules require access growth constants of the function. {To illustrate that nearly optimal parameter-free guarantees are possible, in Section \ref{sec:parallel-bundle} we propose an adaptive parallel bundle method.} We complement our findings with numerical experiments in Section \ref{sec:numerics}. Finally, Section \ref{sec:theory-proofs} presents a broadly applicable proof technique to analyze bundle methods and uses it to establish the theoretical results.


\section{Bundle Methods}\label{sec:bundle}
In this section, we formally define the family of proximal bundle methods that our theory applies to. We present the convergence rates for the classic method with constant stepsizes. Additionally, we introduce and analyze nonconstant stepsize rules that guarantee faster convergence rates.


Proximal bundle methods work by maintaining a model function $f_k\colon \RR^n\rightarrow \RR$ at each iteration $k$ and a current iterate $x_k$. The method computes a candidate for the next iterate as
\begin{equation}\label{eq:zDef} z_{k+1} = \argmin_{z \in \RR^{d}} f_k(z) + \frac{\rho_k}{2}\|z-x_k\|^2. \end{equation}
However, unlike other model-based algorithms, bundle methods do not necessarily move their next iterate to $z_{k+1}$. Instead, it first checks whether the candidate $z_{k+1}$ has at least $\beta\in (0,1)$ fraction of the decrease in objective value that our model $f_k(\cdot)$ predicts, i.e., $\beta(f(x_k) - f_k(z_{k+1})) \leq  f(x_k)-f(z_{k+1})$. If it does, it updates $x_{k+1}=z_{k+1}$ as the next iterate, this is called a {\it Descent Step}. Otherwise the method keeps the iterate the same $x_{k+1}=x_k$  and updates the model function $f_{k+1}$, called a {\it Null Step}.

\begin{algorithm2e}[t]
  \SetAlgoNoLine
	\KwData{$z_0=x_0 \in \RR^n$, $f_0=f(x_0)+\langle g_0, \cdot -x_0\rangle$}
	\Step{}{
	Compute candidate iterate $ \displaystyle z_{k+1} \leftarrow \argmin_{z \in \RR^{d}} f_k(z) + \frac{\rho_k}{2}\|z-x_k\|^2$\;
	\If(\tcp*[f]{Descent step}){$\beta(f(x_k) - f_k(z_{k+1})) \leq  f(x_k)-f(z_{k+1})$}{
	  Set $x_{k+1} \leftarrow z_{k+1}$\;
	}
	\uElse(\tcp*[f]{Null step}){
	  Set $x_{k+1} \leftarrow x_{k}$\;}
	  Update $f_{k+1}$ and $\rho_{k+1}$ without violating Assumption \ref{ass:main}\;
	}
	\caption{Proximal Bundle Method}
	\label{alg:bundle}
\end{algorithm2e}

The proximal bundle method is stated fully in Algorithm~\ref{alg:bundle}. Our analysis does not presume a particular parametrization or form of the models. We only assume that the models satisfy mild assumptions, typical of bundle methods in the literature. To state the assumptions, note the first-order optimality conditions define a subgradient $$s_{k+1}=\rho_k(x_{k}-z_{k+1})\in\partial f_k(z_{k+1}) \quad \text{for each }k \geq 0$$
where $\partial f(x) = \{g \mid f(x') \geq f(x) + \langle g, x'-x\rangle \ \ \forall x'\in\RR^d\}$ denotes the subdifferential of $f$ at $x$. As it is customary, we assume access to a black-box oracle that for a given point $x$ returns $f(x)$ and some subgradient $g(x) \in \partial f(x)$.
\begin{assumption}\label{ass:main}
	Let $\left\{f_k: \RR^d \rightarrow \RR\right\}$ and $\{\rho_k\}$ be the sequence of models and stepsizes used throughout the execution of a bundle method. Assume that for any iteration $k \geq 0$, the next model $f_{k+1}$ and stepsize $\rho_{k+1}$ satisfy the following:
	\begin{enumerate}
		\item \textbf{Minorant.}
		\begin{equation} \label{eq:model-lowerBound}
		f_{k+1}(x)\leq f(x) \qquad  \text{for all }x \in \RR^d \ .
		\end{equation}
		\item \textbf{Subgradient lower bound.} The oracle output $g_{k+1}:= g(z_{k+1}) \in \partial f(z_{k+1})$ satisfies
		\begin{equation} \label{eq:model-subgradf}
		f_{k+1}(x)\geq f(z_{k+1})+\langle g_{k+1}, x-z_{k+1}\rangle \qquad  \text{for all }x \in \RR^d \ .
		\end{equation}
		\item \textbf{Model subgradient lower bound.} After a null step $k$
		\begin{equation} \label{eq:model-subgradfk}
		f_{k+1}(x)\geq f_k(z_{k+1})+\langle s_{k+1}, x-z_{k+1}\rangle \qquad  \text{for all }x \in \RR^d \ .
	  \end{equation}
		\item  \textbf{Nondecreasing stepsize between null steps.} After a null step $k$,
		\begin{equation} \label{eq:parameter-change}
		\rho_{k+1}\geq\rho_k \ .
		\end{equation}
	\end{enumerate}
\end{assumption}

The first two conditions are natural as they ensure that a new model incorporates first-order information from the objective at $z_{k+1}$. The third condition is mild and, intuitively, requires the new model to retain some of the approximation accuracy of the previous model. The last assumption is trivial to enforce algorithmically.

\subsection{Bundle Method Model Function Choices}\label{sec:function_choices}

Several methods for constructing model functions $f_k$ that satisfy~\eqref{eq:model-lowerBound}-\eqref{eq:model-subgradfk} have been considered. In practice, the main consideration lies in weighing the potentially greater per iteration gains from having more complex models against the lower iteration costs from having simpler models.

\paragraph{Full-Memory Proximal Bundle Method.} The earliest proposed bundle methods~\cite{Lemarechal1975,Wolfe1975} rely on using all of the past subgradient evaluations to construct the models as
\begin{equation} \label{eq:full-memory}
f_{k+1}(x) = \max_{j=0\dots k+1} \left\{ f(z_j) + \dotp{g_j, x-z_j} \right\}.
\end{equation}
In this case, solving the quadratically regularized subproblem at each iteration amounts to solving a quadratic programming problem.

\paragraph{Finite Memory Proximal Bundle Method.} Alternatively using cut-aggregation~\cite{Kiwiel1983,Kiwiel1985}, the collection of $k+1$ lower bounds used by~\eqref{eq:full-memory} can be simplified down to just two linear lower bounds. The only two necessary lower bounds are exactly those required by \eqref{eq:model-subgradf} and \eqref{eq:model-subgradfk}. Namely, one could construct the model functions as

\begin{equation} \label{eq:finite-memory}
f_{k+1}(x) = \max\left\{ f_k(z_{k+1})+\langle s_{k+1}, x-z_{k+1}\rangle,\ f(z_{k+1}) + \dotp{g_{k+1}, x-z_{k+1}} \right\}.
\end{equation}
Then the subproblem that needs to be solved at each iteration can be done in closed form, see~ Claim~\ref{claim:closedForm}. Hence the iteration cost using this model is limited primarily by the cost of one subgradient evaluation.

\paragraph{Spectral Bundle Methods.}
Both of the above models rely on constructing piecewise linear models of the objective. For more structure problems, richer models can be constructed. For example, in eigenvalue optimization or more broadly semidefinite programming, better spectral lower bounds can be constructed instead of using simple polyhedral bounds~\cite{Helmberg2000,Oustry2000}. Primal-dual convergence rate guarantees for such spectral bundle methods were recently developed by Ding and Grimmer~\cite{Ding2020}.

\subsection{Iterations, Oracle Queries, Stopping and Inexactness}
In the sequel, we present convergence rates of the bundle method with respect to the number of iterations. Depending on the choice of the model function $f_{k}$, solving the proximal subproblem \eqref{eq:zDef} might require a numerical procedure. Our results treat this step as a black box. In turn, our statements are written in terms of the number of descent and null steps. The sum of these two gives what we call the ``iteration count'' and
ignores any inner-loop iterations needed for solving \eqref{eq:zDef}.

Although the total number inner and outer of iterations might indeed be much higher than our iteration count, our claims about the optimality of the bundle method hold; black-box complexity theory \cite{intro_lect} measures optimality in terms of the number of oracle queries. Regardless of the model function choice, the number of oracle queries matches our iteration count.

{In general, provable stopping criteria for the proximal bundle method are not readily available. One reasonable heuristic lies in stopping once the aggregate $\|s_{k+1}\|$ is small, but this provides no bound on the algorithm's suboptimality. Stopping conditions are possible given more structure (for example, a bound on the distance to optimality or when~\eqref{eq:main} corresponds to the dual of a constrained problem). Alternative approaches like the level bundle method naturally lend themselves to a stopping criteria.}

One shortcoming of our theory is that we do not account for methods where the oracle queries or the solutions of \eqref{eq:zDef} are inexact. There is a vast literature on this topic \cite{kiwiel2006proximal, solodov2003approximations, deOliveira2014, deOliveira2020}. Algorithmic variants that handle inexactness require additional checks in order to guarantee convergence. As extensions of our results to the inexact setting goes beyond the scope of this work; we leave this as an intriguing question for future research.

\subsection{Convergence Rates from Constant Stepsize Choice} \label{sec:constant-step-rates}
We now formalize our convergence theory for the proximal bundle method using any constant choice of the stepsize parameter $\rho_k=\rho$ and any $\beta\in(0,1)$.
These guarantees match those claimed in the first column of Table~\ref{tab:bundle-rates}. After each theorem, we remark on the tuned choice of $\rho$ that gives rise to the claimed rate in the second column of Table~\ref{tab:bundle-rates}. We start by considering the setting where only Lipschitz continuity is assumed. The proof of this result is deferred to Section~\ref{proof:LipschitzGeneral}.
\begin{theorem}[\textbf{Lipschitz}] \label{thm:LipschitzGeneral}
	For any $M$-Lipschitz convex objective function $f$, consider applying the bundle method using a constant stepsize $\rho_k=\rho$. Then for any $0<\epsilon \leq f(x_0)-f^*$, the number of descent steps before an $\epsilon$-minimizer is found is at most
	$$ \frac{2\rho D^2}{\beta\epsilon} + \left\lceil \frac{2\log\left(\frac{f(x_0)-f^*}{\rho D^2}\right)}{\beta} \right\rceil_+ $$
	and the number of null steps is at most
	$$  \frac{48\rho M^2D^4}{\beta(1-\beta)^2\epsilon^3} +\frac{32M^2}{\beta(1-\beta)^2\rho^2D^2} $$
	where $D^2 = \sup_k \dist(x_k,X^*)^2 < \infty$.
\end{theorem}
	It follows from \cite{Ruszczynski2006}[(7.64)] that $D^2 \leq \dist(x_0,X^*)^2 + \frac{2(1-\beta)(f(x_0)-f^*)}{\beta\rho}$. Alternatively, if the level sets of $f$ are bounded, the fact that $f(x_k)$ is non-increasing ensures $D^2 \leq \sup\{\dist(x_k,X^*)^2 \mid f(x)\leq f(x_0)\}$.

	Selecting $\rho= \epsilon/D^2$  gives an overall complexity bound of
	$ O\left(\frac{M^2D^2}{\epsilon^2}\right) $
	and matches the optimal rate for nonsmooth, Lipschitz convex optimization \cite{intro_lect}. 

If instead of Lipschitz continuity of the objective, we assume the objective has Lipschitz gradient, the bundle method adapts to give the following faster rate. We defer the proof to Section~\ref{proof:SmoothGeneral}.
\begin{theorem}[\textbf{Smooth}] \label{thm:SmoothGeneral}
	For any $L$-smooth convex objective function $f$, consider applying the bundle method using a constant stepsize $\rho_k=\rho$. Then for any $0<\epsilon \leq f(x_0)-f^*$, the number of descent steps before an $\epsilon$-minimizer is found is at most
	$$ \dfrac{2\rho D^2}{\beta\epsilon} + \left\lceil \frac{2\log\left(\frac{f(x_0)-f^*}{\rho D^2}\right)}{-\beta} \right\rceil_+ $$
	and the number of null steps is at most
	$$ \dfrac{16(L+\rho)^3}{(1-\beta)^2\rho^3}\left(\frac{2\rho D^2}{\beta\epsilon} + \left\lceil \frac{2\log\left(\frac{f(x_0)-f^*}{\rho D^2}\right)}{\beta} \right\rceil_+ +1\right) $$
	where $D^2 = \sup_k \dist(x_k,X^*)^2 < \infty$.
\end{theorem}
	Selecting $\rho=L$ gives an overall complexity bound of
	$ O\left(\frac{LD^2}{\beta(1-\beta)^2\epsilon}\right).$
	This matches the standard convergence rate for gradient descent \cite{intro_lect}. It is likely that a variant of the bundle method can achieve accelerated rates, either by enforcing additional assumptions about the models $f_{k}$ or by modifying the logic of the algorithm. We leave this as an intriguing open question for future research. We refer the interested reader to \cite{Nesterov2105}, which developed a bundle-like method for smooth functions that outperforms gradient descent in practice. 

Next, we reconsider the settings of Lipschitz continuity and smoothness with additional structure in the form of a H\"older growth bound. We find that the convergence guarantees divide into three regions depending on the growth exponent $p$, whether it is large, equal to, or smaller than $2$. Here $p=2$ is the critical exponent value since the proximal subproblem adds a regularizer with quadratic growth. Regardless, as $p$ decreases, the bundle method converges faster. We defer the proof of the next result to Section~\ref{proof:LipschitzGrowth}.
\begin{theorem}[\textbf{Lipschitz with H\"older growth}] \label{thm:LipschitzGrowth}
	For any $M$-Lipschitz objective function $f$ satisfying the H\"older growth condition~\eqref{eq:holder-growth}, consider applying the bundle method using a constant stepsize $\rho_k=\rho$. Then for any $0<\epsilon \leq f(x_0)-f^*$, the number of descent steps before an $\epsilon$-minimizer is found is at most
	$$ \begin{cases}
	\dfrac{2\rho}{(1-2/p)\beta\mu^{2/p}\epsilon^{1-2/p}} + \left\lceil \dfrac{2\log\left(\frac{f(x_0)-f^*}{(\rho/\mu^{2/p})^{1/(1-2/p)}}\right)}{\beta} \right\rceil_+ & \text{ if } p>2 \\
	\left\lceil \dfrac{2\log\left(\frac{f(x_0)-f^*}{\epsilon}\right)}{\beta\min\{\mu/\rho, 1\}} \right\rceil & \text{ if } p=2 \\
	\left\lceil \dfrac{2\log\left(\frac{(\rho/\mu^{2/p})^{1/(1-2/p)}}{\epsilon}\right)}{\beta} \right\rceil_+ + \dfrac{2\rho(f(x_0)-f^*)^{2/p-1}}{(1-2^{1-2/p})\beta\mu^{2/p}} & \text{ if } 1\leq p<2
	\end{cases} $$
	and the number of null steps is at most
	$$ \begin{cases}
	\dfrac{48\rho M^2}{(1-2/p)\beta(1-\beta)^2\mu^{4/p}\epsilon^{3-4/p}} + \dfrac{32M^2}{\beta(1-\beta)^2\rho(\rho/\mu^{2/p})^{1/(1-2/p)}} & \text{ if } p>2 \\
	\dfrac{16M^2}{\beta(1-\beta)^2\min\{\mu/\rho, 1\}\rho\epsilon}  & \text{ if } p=2 \\
	\dfrac{16M^2}{\beta(1-\beta)^2\rho\epsilon}+ \dfrac{32 M^2}{\beta(1-\beta)^2\rho(\rho/\mu^{2/p})^{1/(1-2/p)}}C & \text{ if } 1\leq p<2
	\end{cases} $$
	with $C=\max\left\{\frac{(f(x_0)-f^*)^{4/p-3}}{(\rho/\mu^{2/p})^{(4/p-3)/(1-2/p)}}, 1\right\}\min\left\{\frac{1}{1-2^{-|4/p-3|}}, \left\lceil\log_2\left(\frac{f(x_0)-f^*}{(\rho/\mu^{2/p})^{1/(1-2/p)}}\right)\right\rceil\right\}$.
\end{theorem}
	When $p=2$, selecting $\rho=\mu$ gives an optimal overall complexity bound of $O(M^2/\mu\epsilon)$. Selecting $\rho=O(\epsilon^{1-2/p})$ matches the optimal rate for Lipschitz optimization with growth exponent $p>2$. When $p=1$, selecting $\rho= O(1/\sqrt{\epsilon})$ minimizes this bound, but the resulting sublinear $O(1/\sqrt{\epsilon})$ rate falls short of the best known rate for sharp, Lipschitz optimization, i.e., linear convergence \cite{polyak1987introduction}. In the next section where we consider nonconstant stepsizes, this disconnect will be remedied and a linear convergence guarantee will follow.

The proof of the next result is deferred to Section~\ref{proof:SmoothGrowth}.
\begin{theorem}[\textbf{Smooth with H\"older growth}] \label{thm:SmoothGrowth}
	For any $L$-smooth objective function $f$ satisfying the H\"older growth condition~\eqref{eq:holder-growth}, consider applying the bundle method using a constant stepsize $\rho_k=\rho$. Then for any $0<\epsilon \leq f(x_0)-f^*$, the number of descent steps before an $\epsilon$-minimizer is found is at most
	$$ \begin{cases}
	\dfrac{2\rho}{(1-2/p)\beta\mu^{2/p}\epsilon^{1-2/p}} + \left\lceil \dfrac{2\log\left(\frac{f(x_0)-f^*}{(\rho/\mu^{2/p})^{1/(1-2/p)}}\right)}{\beta} \right\rceil_+ & \text{ if } p>2 \\
	\left\lceil \dfrac{2\log\left(\frac{f(x_0)-f^*}{\epsilon}\right)}{\beta\min\{\mu/\rho, 1\}} \right\rceil & \text{ if } p=2
	\end{cases} $$
	and the number of null steps is at most
	$$ \begin{cases}
	\dfrac{16(L+\rho)^3}{(1-\beta)^2\rho^3}\left(\dfrac{2\rho}{(1-2/p)\beta\mu^{2/p}\epsilon^{1-2/p}} + \left\lceil \dfrac{2\log\left(\frac{f(x_0)-f^*}{(\rho/\mu^{2/p})^{1/(1-2/p)}}\right)}{\beta} \right\rceil_+ +1\right) & \text{ if } p>2 \\
	\dfrac{16(L+\rho)^3}{(1-\beta)^2\rho^3}\left\lceil \dfrac{2\log\left(\frac{f(x_0)-f^*}{\epsilon}\right)}{\beta\min\{\mu/\rho, 1\}} \right\rceil & \text{ if } p=2 \ .
	\end{cases} $$
\end{theorem}
  Selecting $\rho=L$ gives an overall complexity bound matching gradient descent \cite{intro_lect}.

\subsection{Convergence Rates from Improved Stepsize Choice} \label{sec:improved-step-rates}
Picking $\rho_k$ to vary throughout the execution of the bundle method allows for stronger convergence guarantees. These rates are formalized in the following pair of theorems that consider settings with and without H\"older growth. In the latter case, we find that our stepsize choice removes the need for piecewise guarantees around growth exponent $p=2$, which notably simplifies the statement of our results.

Intuitively, the stepsize choices are aim to mimic the following idealistic (and impractical) stepsize rule that naturally arises from our theory
\begin{equation} \label{eq:ideal-stepsize}
\rho_k = \frac{f(x_k)-f(x^*)}{\|x_k-x^*\|^2} \ .
\end{equation}
In Section~\ref{sec:theory-proofs} after introducing our main lemmas, we discuss the convergence bounds that would result from approximating this stepsize. The proof techniques we develop could be extended to study other nonconstant stepsizes. For instance, stepsizes that shrink/grow polynomial with the number of iterations, mirroring those used for subgradient methods. The analysis of such schemes is beyond the scope of this work.

\if\siam1
The proofs of the following two results are similar to those with constant stepsizes; we refer the interested reader to the technical report \cite{diaz2021optimal} for detailed arguments.
\else
The proof of the next two Theorems are presented in Sections~\ref{proof:OPT-LipschitzGeneral} and \ref{proof:OPT-LipschitzGrowth}, respectively.
\fi
\begin{theorem}[\textbf{Lipschitz}] \label{thm:OPT-LipschitzGeneral}
	For any $M$-Lipschitz objective function $f$, consider applying the bundle method using the stepsize policy
	\begin{equation} \label{eq:OPTgeneral-stepsize}
	\rho_k = (f(x_k)-f^*)/D^2
	\end{equation}
	with any choice of $D^2 \geq \sup\{\dist(x,X^*)^2 \mid f(x)\leq f(x_0)\}$. Then for any $0<\epsilon \leq f(x_0)-f^*$, the number of descent steps before an $\epsilon$-minimizer is found is at most
	$$ \left\lceil\frac{2\log\left(\frac{f(x_0)-f^*}{\epsilon}\right)}{\beta}\right\rceil $$
	and the number of null steps is at most
	$$ \left(\frac{1}{1-(1-\beta/2)^2}\right)\frac{8M^2D^2}{(1-\beta)^2\epsilon^2} \ . $$
\end{theorem}
\begin{theorem}[\textbf{Lipschitz with H\"older growth}] \label{thm:OPT-LipschitzGrowth}
	For any $M$-Lipschitz objective function $f$ satisfying the H\"older growth condition~\eqref{eq:holder-growth}, consider applying the bundle method using the stepsize policy
	\begin{equation} \label{eq:OPTholder-stepsize}
	\rho_k = \mu^{2/p}(f(x_k)-f^*)^{1-2/p}.
	\end{equation}
	Then for any $0<\epsilon \leq f(x_0)-f^*$, the number of descent steps before an $\epsilon$-minimizer is found is at most
	$$ \left\lceil\frac{2\log\left(\frac{f(x_0)-f^*}{\epsilon}\right)}{\beta}\right\rceil $$
	and the number of null steps is at most
	$$ \begin{cases}
	\left(\dfrac{1}{1-(1-\beta/2)^{2-2/p}}\right)\dfrac{8M^2}{(1-\beta)^2\mu^{2/p}\epsilon^{2-2/p}} & \text{ if } p>1 \\
	\dfrac{8M^2}{(1-\beta)^2\mu^2}\left\lceil\dfrac{2\log\left(\frac{f(x_0)-f^*}{\epsilon}\right)}{\beta}\right\rceil & \text{ if } p=1 \ .
	\end{cases}$$
\end{theorem}


\section{The parallel bundle method}\label{sec:parallel-bundle}
Several works \cite{kiwiel1990proximity, bonnans1995family, lemarechal1997variable, de2016doubly}  have proposed practical step size rules that update $\rho_{k}$ on the fly using the objective values of the iterates. Such rules enjoy asymptotic convergence guarantees and tend to perform well in practice. However, to our knowledge, none of them come equipped with efficiency rates. In this section, we take an alternative approach to tackle tuning: we run a logarithmic number of parallel instances of the bundle method, with different constant step sizes, inspired by the ideas of~\cite{RenegarGrimmer2021}. The instances communicate at the end of each iteration and update their models based on each other's progress. We show that this procedure achieves optimal convergence rates, up to the cost of running a logarithmic number of instances which can be mitigated through parallelization. Unlike the step sizes proposed in the previous section, { this parallel scheme provides a provably near optimal proximal bundle method that is parameter-free. This scheme's primary drawback lies in the nontrivial overhead of running several instances at once.} 

The core observation behind our parallel method is that if we run the bundle method with the nonconstant step size rules~\eqref{eq:OPTgeneral-stepsize} and~\eqref{eq:OPTholder-stepsize}, then the step sizes $\rho_{k}$ will lie in the interval $$\rho_k \in \left[c_{1}\epsilon, c_{2}\epsilon^{-1}\right], \quad \text{for some fixed constants }c_{1}, c_{2} > 0, $$ before the algorithm finds an $\epsilon$-minimizer.
  Intuitively, our plan will be to run a modest number of instances of the bundle method, each with a different constant step size in the interval above, and combine their progress to achieve a faster convergence rate.

As input, we only assume the following are given: a lower bound $\bar{\rho}$ and an upper bound $2^J\bar{\rho}$ on the range of step sizes to consider, where $J \in \NN$ controls the number of instances we run in parallel.  Provided our step size rules~\eqref{eq:OPTgeneral-stepsize} and~\eqref{eq:OPTholder-stepsize} lie in this interval,
$$ \rho_k \in \left[\bar{\rho}, 2^J\bar{\rho}\right] \ ,$$
we are able to recover nearly optimal convergence rates. Notice that the interval $[\bar{\rho},2^J\bar{\rho}]$ can span the whole range of step sizes needed for our H\"older growth analysis by setting $\bar{\rho}=O(\epsilon)$ and $J= O(\log_{2}(1/\epsilon^2))$. Our resulting convergence guarantees only depend logarithmically on the size of this interval (a cost which can be mitigated through parallelization), so $\bar{\rho}$ and $2^J\bar{\rho}$ can bet set generously at {relatively} little cost.

\paragraph{Description of the algorithm.} We now informally outline the parallel bundle method and refer the reader to Algorithm \ref{alg:cent} for detailed pseudo-code. Recall that we use $g(z)$ to denote the subgradient oracle evaluated at $z \in \RR^{d}$. We propose running $J$ instances of the bundle method in parallel
  and enumerate them as $\{0,\dots J-1\}$. The $j$th instance uses a constant step size $\rho^{(j)} = 2^j\bar{\rho}$. We denote its iterates by $x^{(j)}_k$ and its model functions by $f^{(j)}_k$. Each instance $j$ proceeds as normal with the only modification being that after it takes a descent step, the algorithm checks if there exists an instance $j'$ that has an iterate with even lower objective value $f(x^{(j')}_{k}) < f(x^{(j)}_{k+1})$. If such an improvement exists, instance $j$ updates
  $$\begin{cases} x^{(j)}_{k+1} &\leftarrow x^{(j')}_{k} \\
    f^{(j)}_{k+1}(z) &\leftarrow f(x^{(j')}_{k})+\langle g(x^{(j')}_{k}), z-x^{(j')}_{k}\rangle \end{cases}$$
  and then proceeds.

  For the sake of analysis, we assume that each parallel instance of the bundle method operates \emph{synchronously}; that is, all instances have to finish the current iteration before any of them can proceed to the next one. This process can be implemented either sequentially or in parallel by cycling through the bundle method instances and computing one iteration for each before repeating. An asynchronous variant of this procedure could be analyzed as well, using similar techniques as those in \cite{RenegarGrimmer2021}. However, this is beyond the focus of this work.

  The choice to use powers of two defining $\rho^{(j)}$ is arbitrary. For the numerical experiments in Section~\ref{sec:numerics}, we use powers of $10$ and $10000$ demonstrating the effectiveness of this scheme even when using a sparse selection of sample step sizes.

\begin{algorithm2e}[t]
  
  \SetAlgoNoLine
      \KwData{$x_0 \in \RR^n$, $f_0=f(x_0)+\langle g(x_0), \cdot -x_0\rangle, \bar\rho \in \RR_{+}, J \in \NN_{+}$}
    \For{$j \in \{0, \dots, J-1\}$}{
      Set $\rho^{(j)} \leftarrow  2^j \cdot \bar \rho$, $x_{0}^{(k)} \leftarrow x_{0}$, $\texttt{descent}^{(j)} = \texttt{true}$\;
    }
    \BlankLine
    \Step(){}{
      \ParFor(\tcp*[f]{Each instance runs a step}){$j \in \{0, \dots, J-1\}$} {
        Compute $ \displaystyle z_{k+1}^{(j)} \leftarrow \argmin_{z \in \RR^{d}} f_{k}^{(j)} (z) + \frac{\rho^{(j)}}{2}\|z-x_{k}^{(j)}\|^2$\;
        \If(\tcp*[f]{Descent step}){$\beta(f(x_{k}^{(j)}) - f_{k}^{(j)} (z_{k+1}^{(j)})) \leq  f(x_{k}^{(j)})-f(z_{k+1}^{(j)})$}{
          Update $x_{k+1}^{(j)} \leftarrow z_{k+1}^{(j)}$, and $\texttt{descent}^{(j)} \leftarrow \texttt{true}$\;
        }
        \uElse(\tcp*[f]{Null step}){
          Update $x_{k+1}^{(j)}\leftarrow x_{k}^{(j)}$ and $\texttt{descent}^{(j)} \leftarrow \texttt{false}$\;
        }
        Update $f_{k+1}^{(j)}$ without violating Assumption \ref{ass:main}.
      }
      \BlankLine
      Find best iterate $j^{*} \leftarrow \argmin_{j \in \{0, \dots, J-1\}} f(x_{k}^{(j)})$\;
      \BlankLine
      \ParFor(\tcp*[f]{Communication round}){$ j\in \{0, \dots, J-1\}$} {
        \If{$\texttt{\upshape descent}^{(j)}$ \textbf{\upshape and} $f(x_{k}^{(j^{*})}) < f(x_{k+1}^{(j)})$}{
          Update $x_{k+1}^{(j)} \leftarrow x_{k}^{(j^{*})}$, and $f_{k+1}^{(j)} \leftarrow f(x_{k}^{(j^{*})}) + \dotp{g(x_{k}^{(j^{*})}), \cdot - x_{k}^{(j^{*})}}$\;
        }
      }
    }
  \caption{Parallel bundle method}
  \label{alg:cent}
\end{algorithm2e}
\subsection{Convergence Rates for the Parallel Bundle Method}
First, we remark that all of our previous convergence theory for constant stepsizes (Theorems~\ref{thm:LipschitzGeneral}, \ref{thm:SmoothGeneral}, \ref{thm:LipschitzGrowth}, and \ref{thm:SmoothGrowth}) immediately apply to the Parallel Bundle Method fixing $\rho=2^j\bar{\rho}$ for any $j\in\{0,\dots J-1\}$. This follows as our convergence theory on relies on a lemma ensuring sufficient decrease at each descent step (Lemma~\ref{lem:descent-prox-gap}) and the new case of a bundle method restarting at another method's lower objective value iterate can only further improve on this decrease. Hence any individual instance of the bundle method with $\rho^{(j)}=2^j\bar{\rho}$ in our parallel scheme will converge at least as fast as Theorems~\ref{thm:LipschitzGeneral}, \ref{thm:SmoothGeneral}, \ref{thm:LipschitzGrowth}, and \ref{thm:SmoothGrowth} guarantee it would converge on its own.

Further and more importantly, when our nonconstant stepsize rules~\eqref{eq:OPTgeneral-stepsize} and~\eqref{eq:OPTholder-stepsize} lie in the interval $[\bar{\rho},2^J\bar{\rho}]$, we find that their convergence theory (Theorems~\ref{thm:OPT-LipschitzGeneral} and \ref{thm:OPT-LipschitzGrowth}) also extends to our parallel algorithm. This is formalized by the following theorem. We defer the proof of this result to Section~\ref{proof:parallel}.

\begin{theorem} \label{thm:OPT-ParallelRate}
  For any $M$-Lipschitz objective function $f$ that satisfies the H\"older growth condition~\eqref{eq:holder-growth}, consider applying the Parallel Bundle Method with stepsizes (Algorithm~\ref{alg:cent}) with input $x_{0} \in \RR^{d}$, $\bar \rho \in \RR_{+}$, and $J \in \NN_{+}$. Then for any $0< \epsilon\leq f(x_0)-f(x^*)$, if
  $$ \bar{\rho} \leq \frac{1}{4}\mu^{2/p}\min\{\epsilon^{1-2/p},(f(x_0)-f(x^*))^{1-2/p}\} $$
  and
  $$ J\geq \log_2\left(\frac{\mu^{2/p}(\max\{\epsilon^{1-2/p},(f(x_0)-f(x^*))^{1-2/p}\}}{4\bar{\rho}}\right)\ ,$$
  then the iterate with minimum objective $x_{\text{\upshape best}}$ becomes an $\epsilon$-minimizer within the first
  $$ \begin{cases}
	\left(\dfrac{2}{1-(1-\beta/2)^{2-2/p}}\right)\dfrac{64M^2}{(1-\beta)^2\mu^{2/p}\epsilon^{2-2/p}} + 2\left\lceil \dfrac{2\log(\frac{f(x_0)-f^*}{\epsilon})}{\beta}\right\rceil  & \text{ if } p>1 \\
	2\left(\dfrac{64M^2}{(1-\beta)^2\mu^{2}} +1\right)\left\lceil \dfrac{2\log(\frac{f(x_0)-f^*}{\epsilon})}{\beta}\right\rceil & \text{ if } p=1
  \end{cases} $$
  iterations, {each utilizing $J$ subgradient oracle evaluations}.
\end{theorem}

  {We reiterate that the total number of oracle calls required by the parallel method is given by the product between number of iterations and $J$. This limits the practicality for applications where oracle costs are expensive or cannot be easily parallelized.}

  Nonetheless, these rates nearly match the optimal lower bounds for nonsmooth Lipschitz optimization, up to small constants and an additive logarithmic term when $p>1$. For example, under quadratic growth when $p=2$, selecting $\bar \rho  \leq \mu/4$ and $J \geq \log_2(\mu/4\bar{\rho})$, an $\epsilon$-minimizer is found within
  $$ O\left(\frac{M^2}{\mu\epsilon}\right)$$
  iterations (each utilizing $J$ subgradient evaluations, which in this case, is constant with respect to $\epsilon$).
  Under sharp growth $p=1$, selecting $\bar \rho  \leq \frac{\mu^2}{4(f(x_0)-f(x^*))}$ and $J \geq \log_2(\mu^2/4\bar{\rho}\epsilon)$ yields a convergence rate of
  $$ O\left(\frac{M^2}{\mu^2}\log(1/\epsilon)\right) $$
  with each step utilizing a logarithmic number $J=O(\log(1/\epsilon))$ of subgradient evaluations.
  Critically, these convergence rates only depend on $\bar \rho$ and $J$ through the (parallelizable) cost of updating the $J$ bundle method instances at each iteration.


\section{Numerical experiments}\label{sec:numerics}

In this section, we present three examples {with a focus on illustrating our theory for the bundle method rather than on the numerical practicality of any proposed stepsize scheme}. These experiments were implemented in Julia; see \url{https://github.com/mateodd25/proximal-bundle-method} for the source code. For all the numerical experiments, we use a finite-memory model function \eqref{eq:finite-memory}. The reason is two-fold: first, we want to illustrate our theoretical efficiency rates, which are independent of the model $f_{k}$, and, second,
  with this choice, the subproblem \eqref{eq:zDef} has a closed-form solution, thus reducing the computational footprint.

\subsection{Sharp linear regression}

\begin{figure}
	\centering
	\begin{subfigure}[b]{0.43\textwidth}
		\centering
		\includegraphics[width=\textwidth]{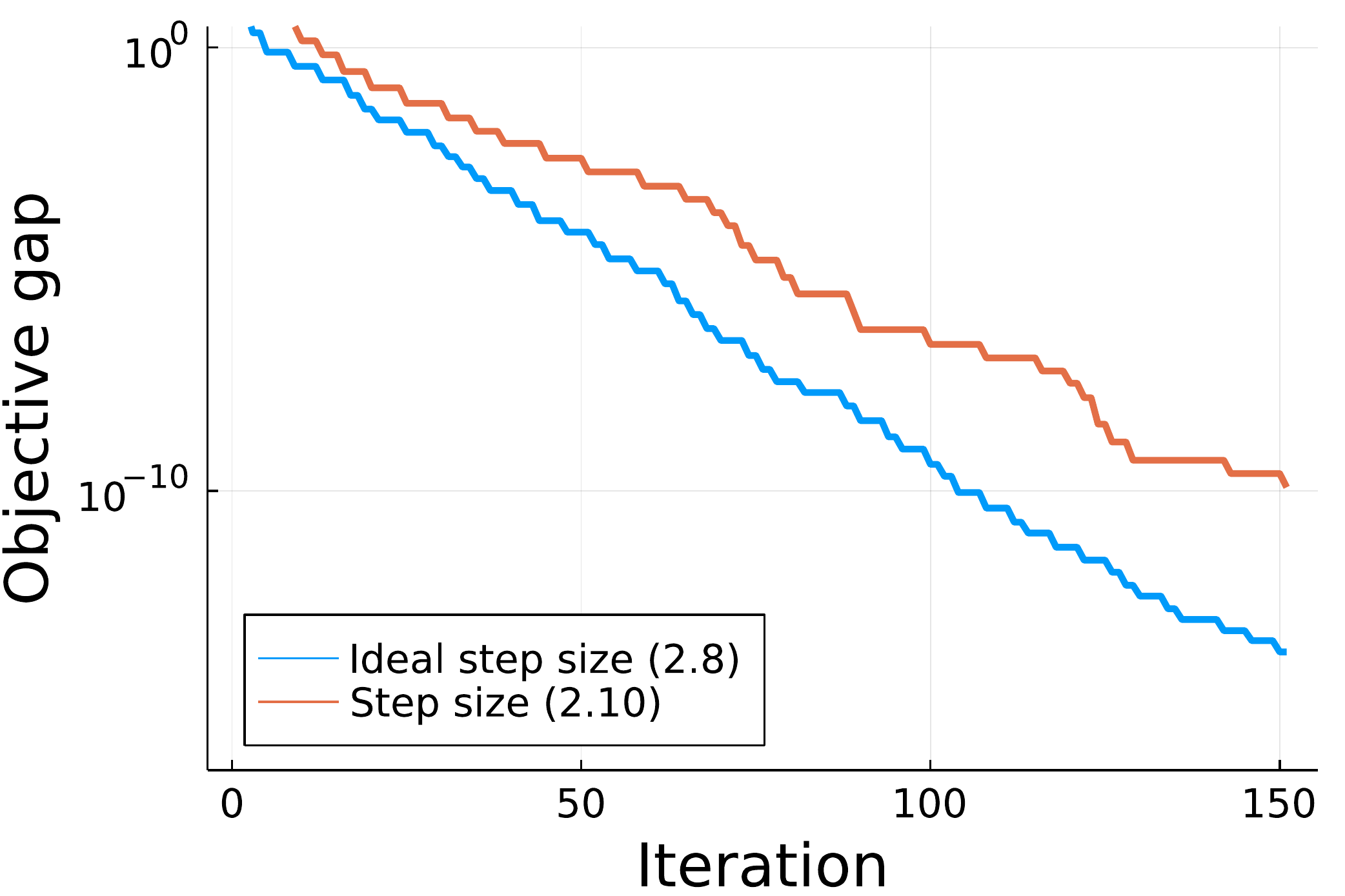}
	\end{subfigure}
	\hfill
	\begin{subfigure}[b]{0.43\textwidth}
		\centering
		\includegraphics[width=\textwidth]{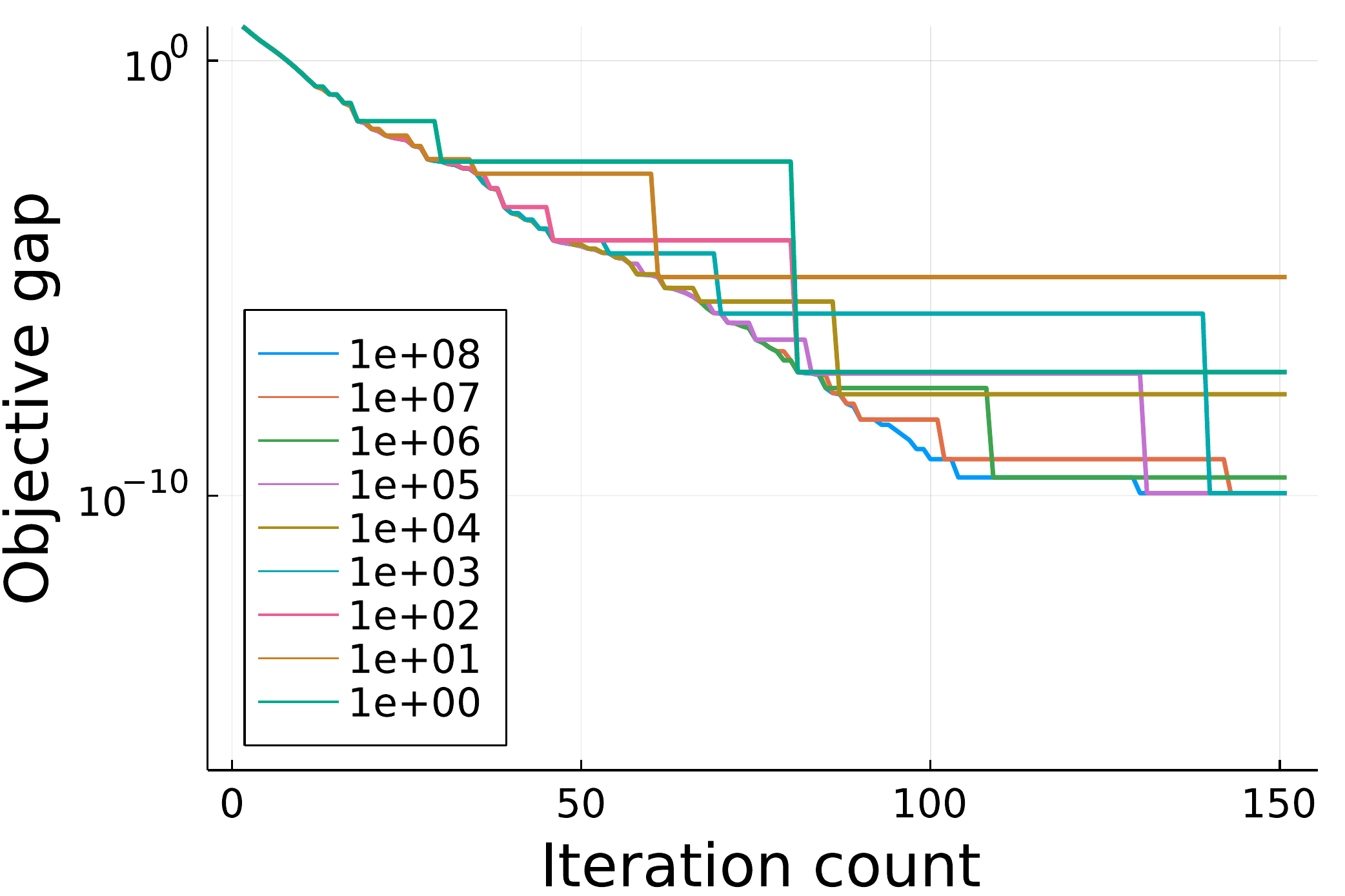}
	\end{subfigure}
	\caption{Objective gap against iteration count: using ideal stepsize \eqref{eq:ideal-stepsize} (left) and using the parallel bundle method (right), plotting each instance deployed with stepsizes from $10^0,\dots,10^8$.}
	\label{fig:objectives}
\end{figure}

\begin{figure}
	\centering
	\begin{subfigure}[b]{0.43\textwidth}
		\centering
		\includegraphics[width=\textwidth]{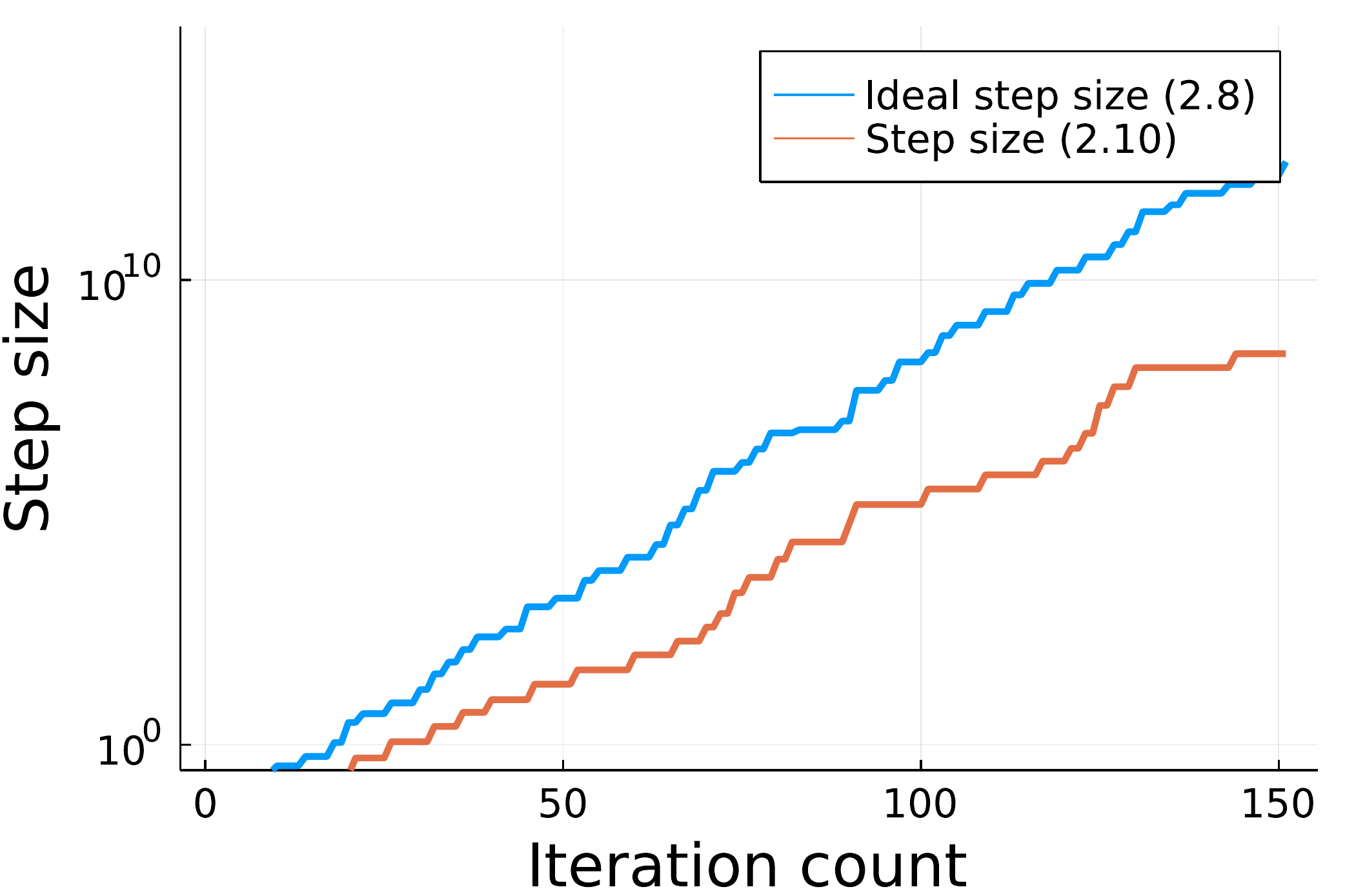}
	\end{subfigure}
	\hfill
	\begin{subfigure}[b]{0.43\textwidth}
		\centering
		\includegraphics[width=\textwidth]{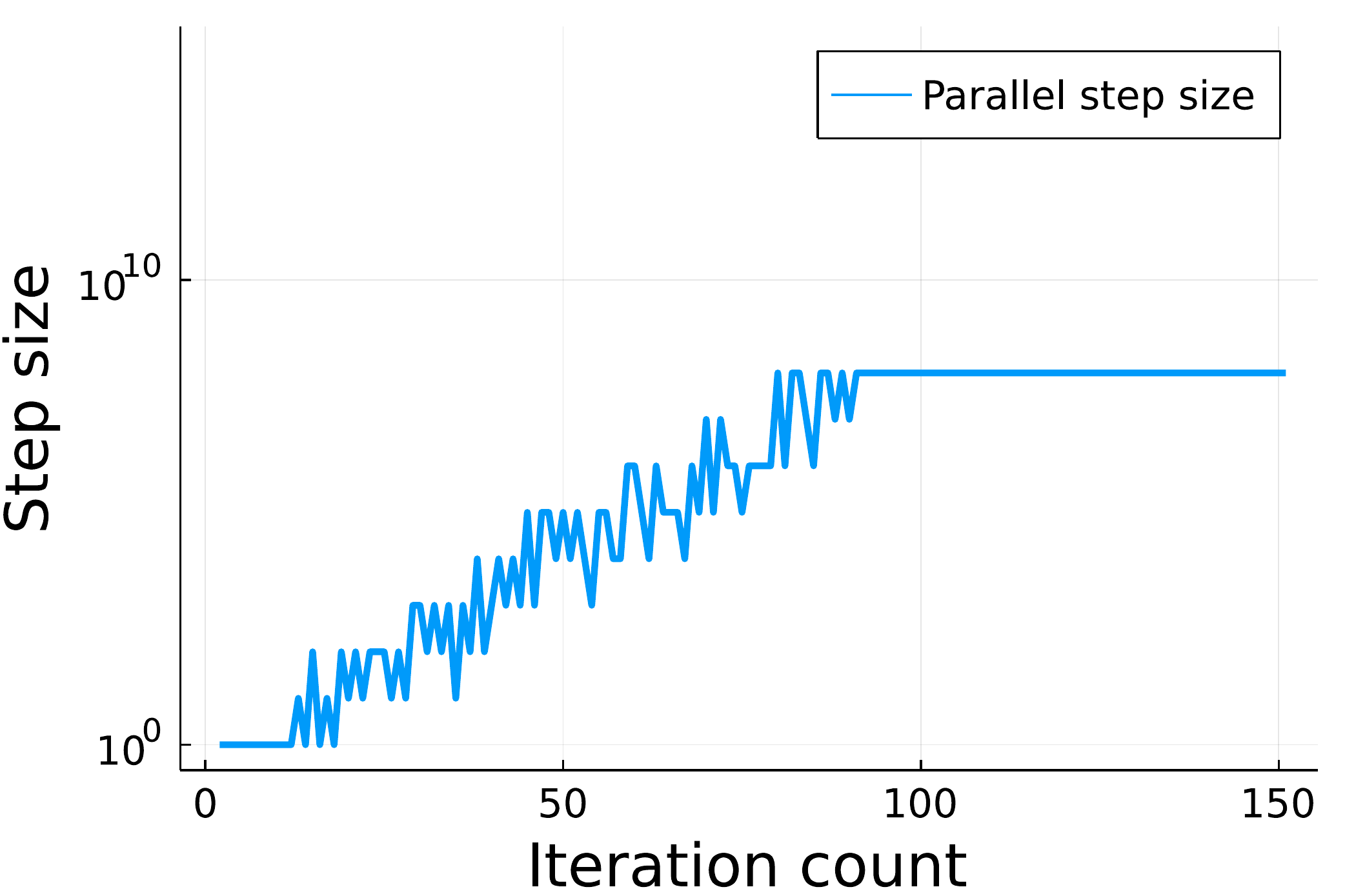}
	\end{subfigure}
	\caption{Stepsize against iteration count: using ideal stepsize \eqref{eq:ideal-stepsize} (left) and using the parallel bundle method (right).}
	\label{fig:stepsizes}
\end{figure}

The first experiment aims to exemplify the fast convergence of the bundle method under sharp growth. We consider a simple linear regression problem of the form
$$ \min_{x \in \RR^d} f(x) :=  \|Ax -b\|$$
where $A \in \RR^{n \times d}$ is a matrix and $b = A x^\star$ for a fixed $x^\star$. This problem is equivalent to the classic least-squares problem after taking squares. Yet, without the square it is well known that for Gaussian matrices, $(A)_{ij} \sim N(0, \frac{1}{\sqrt{n}})$, this function is sharp and Lipchitz continuous provided $n$ is large enough \cite{charisopoulos2021low}.

We generate a random Gaussian matrix $A \in \RR^{100 \times 50}$ and random solution $x^\star \sim N(0, I_d)$. We run three algorithms: two proximal bundle methods with step sizes \eqref{eq:ideal-stepsize} and \eqref{eq:OPTholder-stepsize}, respectively, and the parallel bundle method described in Section \ref{sec:parallel-bundle}. The step sizes \eqref{eq:ideal-stepsize} and \eqref{eq:OPTholder-stepsize} are impractical since they require knowing the minimum value, and the sharpness constant. However, the theoretical analysis shows that they yield optimal convergence rates, and so we use them as a baseline.
 The parallel bundle method uses nine parallel instances with step sizes in $\rho \in \{1, 10, \dots,10^8\}.$ We let both methods run for $150$ iterations.

Figure \ref{fig:objectives} displays the best objective gap $f - \min f$ so far against the iteration count for the three methods. On the other hand, Figure \ref{fig:stepsizes} shows the step size used at each iteration. For the parallel bundle method, we display the step size used by the last instance to reduce the best objective value seen.
As the theory predicts the convergence of all three methods is linear. The bundle method with step size \eqref{eq:ideal-stepsize} exhibits steady progress and reaches an objective gap of $10^{-15}$, while the parallel version slows down around $100$ iterations and only achieves a loss of gap of $10^{-10}$. This behavior is explained by the step size plots. The impractical step sizes behave like $\mu^{2}/(f(x_{k})-\min f),$ and so they keep increasing as the iterates converge. Figure~\ref{fig:stepsizes} plots how the parallel algorithm roughly emulates \eqref{eq:ideal-stepsize}, until it exceeds the maximum step size that the parallel bundle method can use, i.e., $10^8$. After which, the instance with step size $10^8$ consistently leads the method's progress, albeit sublinearly. {The fact that the parallel method adapted to follow a stepsize scheme similar to~\eqref{eq:ideal-stepsize} provides numerical support for our proposed ideal stepsize.}

\begin{wrapfigure}{l}{5.6cm}
  \includegraphics[width=5.6cm]{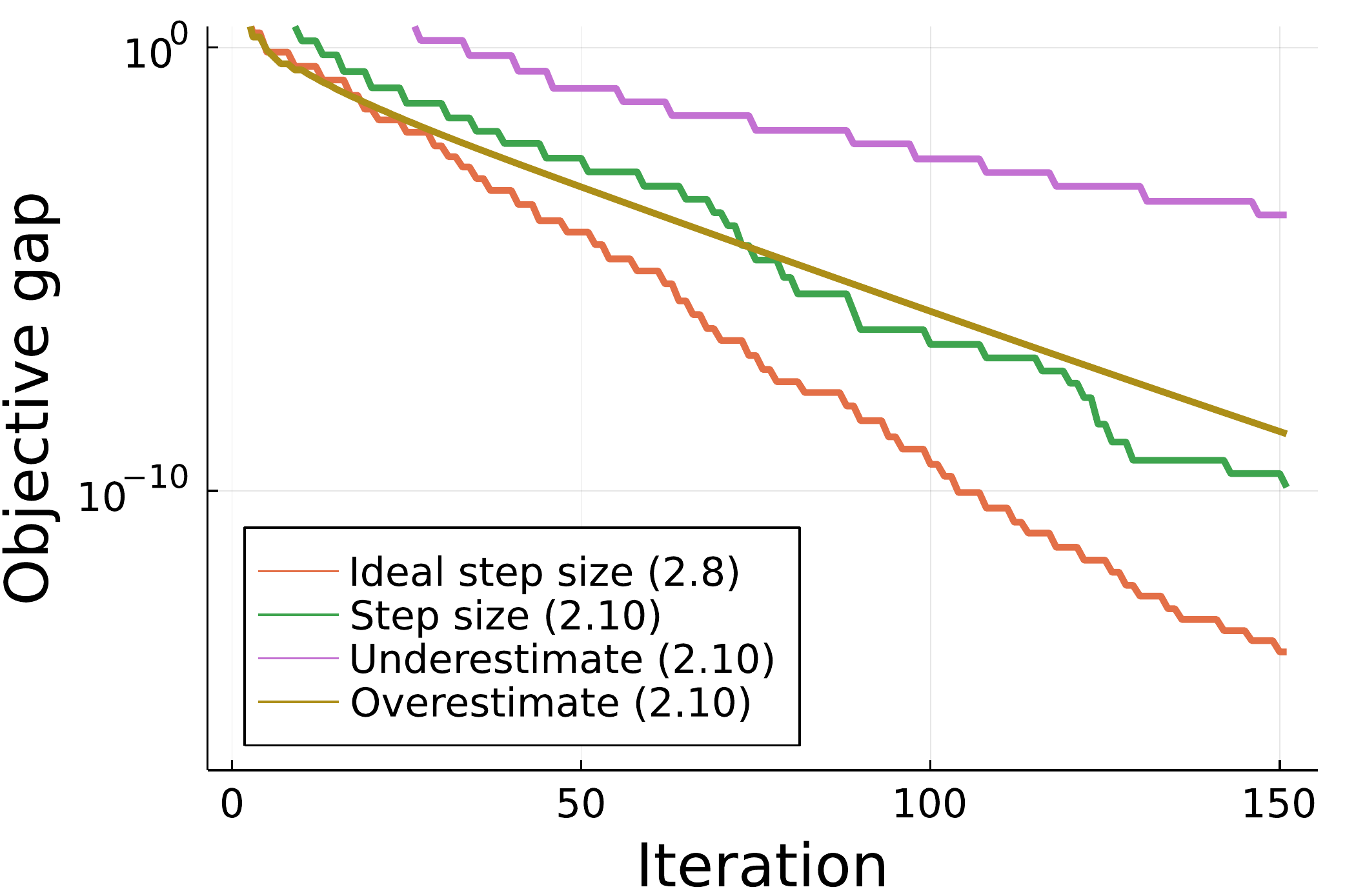}
  \caption{Objective gap with misspecified sharpness constant. See text for description.}\label{fig:misspecification}
  \vspace{-15pt}
\end{wrapfigure}

\paragraph{Parameter misspecification.} \ {The theoretical analysis suggests that the convergence of proximal bundle methods degrades gracefully with the misspecification of the growth constants; see the discussion after Lemma~\ref{lem:prox-gap-bound} below. To corroborate this, we perform the same sharp regression experiment, but we underestimate by a factor of $1/3$ and overestimate by a factor of $3$ the sharpness constant. We use these estimates with step size \eqref{eq:OPTholder-stepsize} and compare them against the previous executions with the correct estimate and the ideal step size \eqref{eq:ideal-stepsize}.
Figure~\ref{fig:misspecification} displays the results: even with a misspecified $\mu$, the method still exhibits linear convergence, although at a slower rate.}

\subsection{Support Vector Machine}

To illustrate the adaptive features of the parallel bundle method we consider the standard Support Vector Machine (SVM) formulation: we are giving datapoints $(x_1, y_1), \dots (x_n, y_n)$ with $x_i \in \RR^d$ and $y_i \in \{\pm 1\}$ and our goal is to solve
\begin{equation}
\label{eq:3}
\min_{w \in \RR^d} \frac{1}{n} \sum \max\left\{0, 1 - y_i\dotp{w, x_i}\right\} + \frac{\lambda}{2} \|w\|^2
\end{equation}
where $\lambda \in \RR$ is a fixed constant. This problem is not smooth due to the first term. For this experiment we compare against a subgradient method based on Pegasos \cite{shalev2011pegasos}, a state-of-the art solver for SVM. Our vanilla implementation of the parallel bundle method is not tuned for efficiency and does not aim to be competitive with commercial solvers. Instead, we aim to show that an out-of-the-box implementation is immediately comparable to a specialized first-order method for this problem.

We generate SVM problems using three datasets from the LIBSVM Binary Classification Database \cite{libsvm}. In particular, we use \texttt{colon-cancer}, \texttt{duke}, and \texttt{leu}.\footnote{We refer the reader to LIBSVM for the origin of each of these datasets.} We preprocess the data by deleting empty features, normalizing the features, and adding an extra component $x_k = (x_k, 1)$ to allow for affine functions.

The implementation of the subgradient algorithm updates
$$w_{k+1} \leftarrow (1-\eta_k \lambda)w_k + \eta_k \sum_{i=1}^{n}\mathbf{1}\{1 \leq y_i\dotp{w_k, x_i}\} y_i x_i$$
where $\eta_k = \frac{1}{\lambda k}$ and $\mathbf{1}\{\cdot\}$ is one if $\cdot$ holds true and zero otherwise. This is analogous to Pegasos with the exception that it does full instead of stochastic subgradient evaluations. Knowledge of $\lambda$ is necessary for the implementation of this method.

\begin{figure}
	\centering
	\begin{subfigure}[b]{0.3\textwidth}
		\centering
		\includegraphics[width=\textwidth]{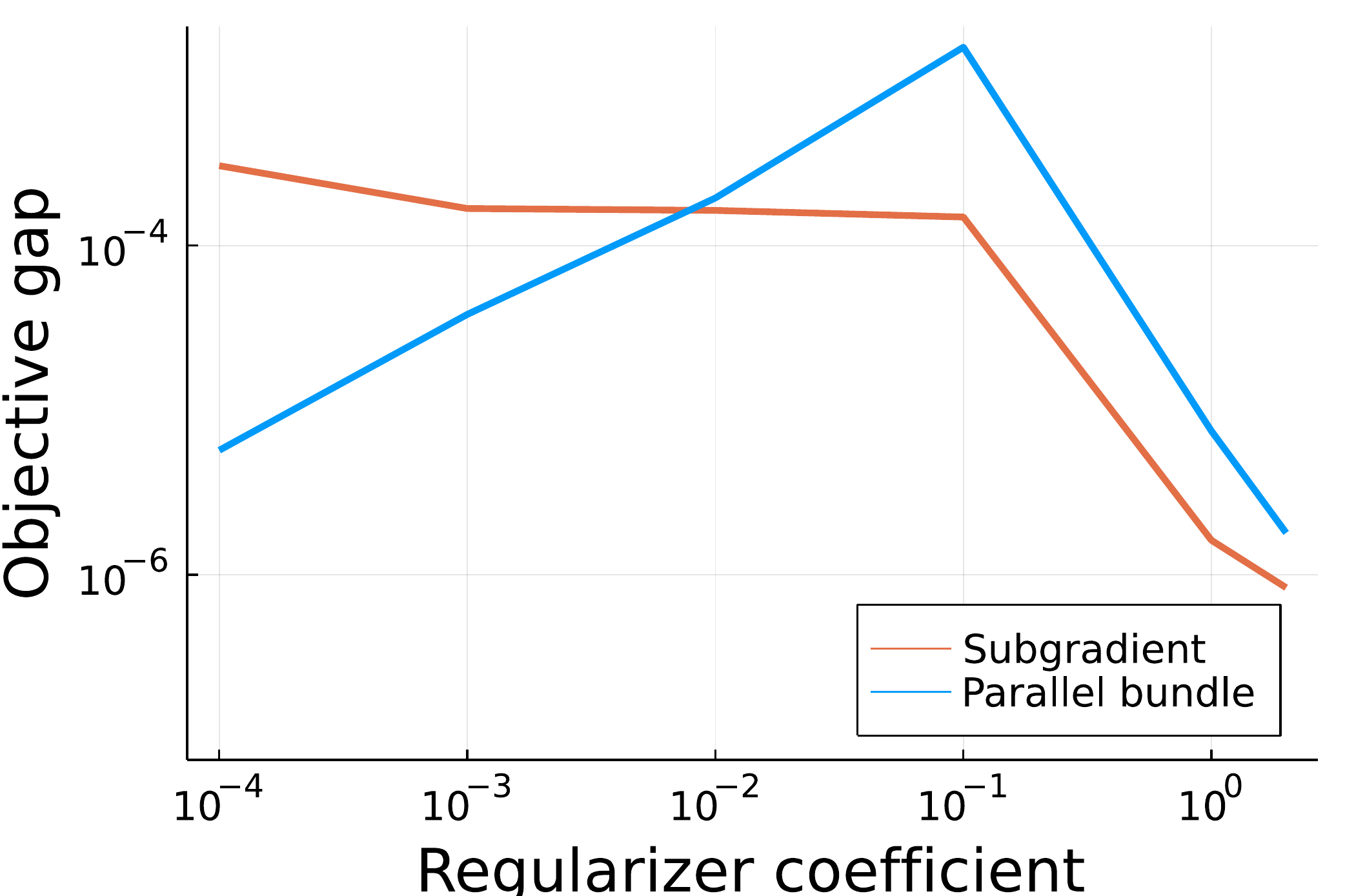}
	\end{subfigure}
	\hfill
	\begin{subfigure}[b]{0.3\textwidth}
		\centering
		\includegraphics[width=\textwidth]{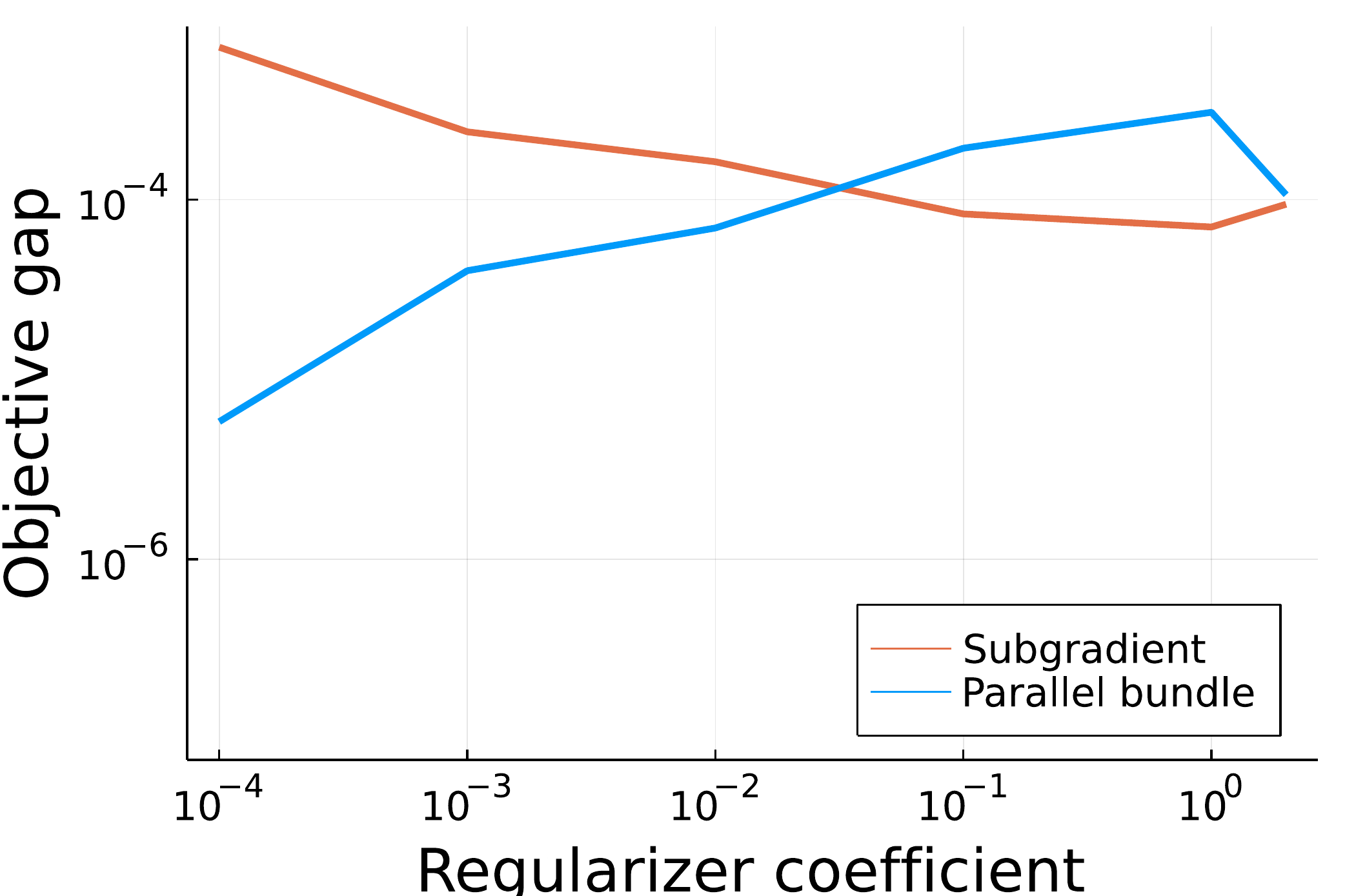}
	\end{subfigure}
	\hfill
	\begin{subfigure}[b]{0.3\textwidth}
		\centering
		\includegraphics[width=\textwidth]{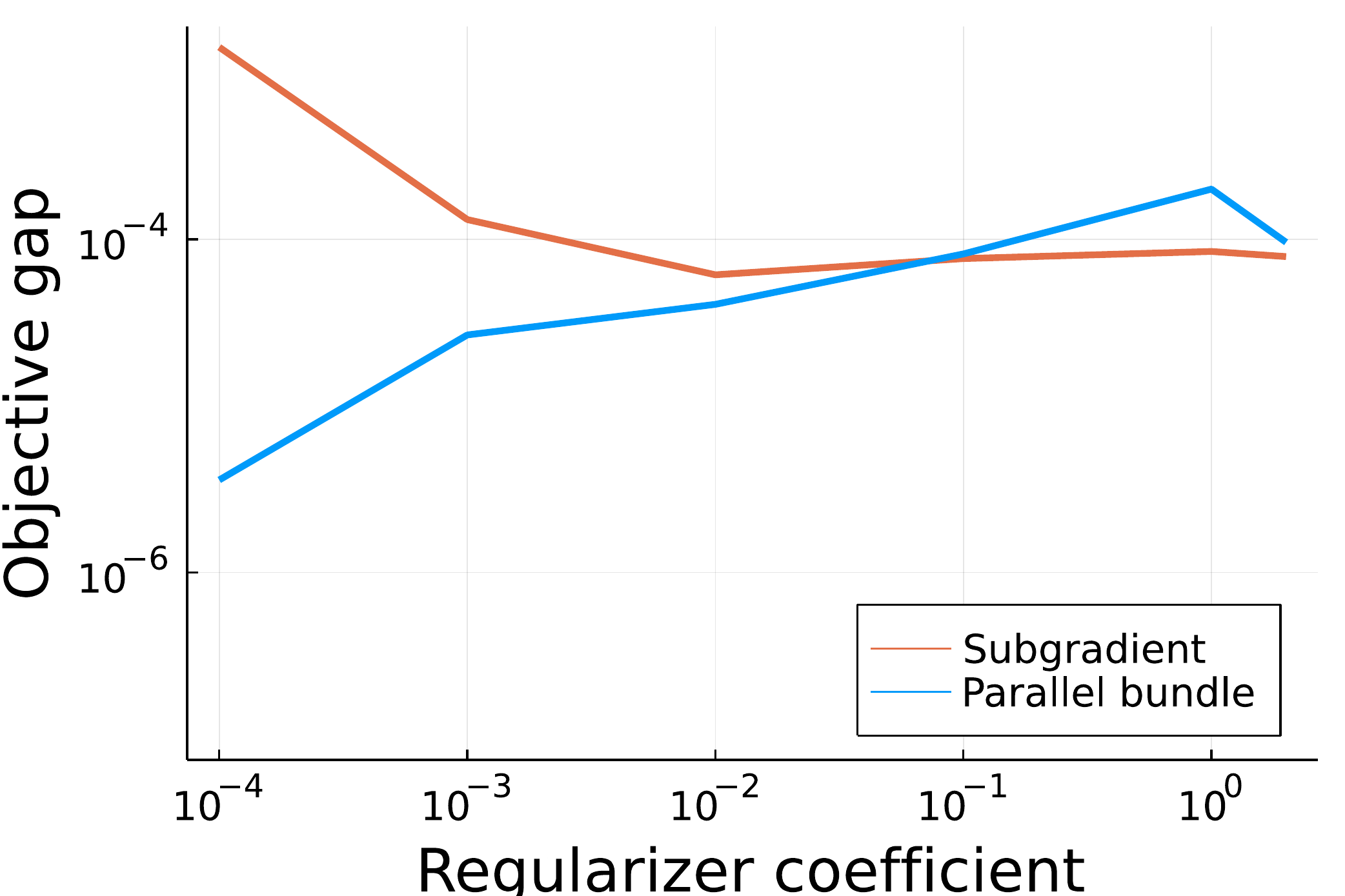}
	\end{subfigure}
	\caption{Log-log plot of best objective gap against coefficient $\lambda$ for the three problems, solved by a subgradient method and by the parallel bundle method: \texttt{colon-cancer} (left), \texttt{duke} (center), and \texttt{leu} (right); see text for details.}
	\label{fig:svm}
\end{figure}

For the parallel bundle method, we use stepsizes three instances with constant stepsizes $\rho\in \{10^{-9} \cdot 10^{4j} \mid j = 0, 1, 2\}.$ We run the bundle method for $2000$ iterations and the subgradient method for $6000$ thousand iterations, that way both methods make the same number of oracle calls. We measure the best objective gap $ f(x_{k}) - \min f$ so far. To compute the minimum we use Gurobi with accuracy set to $10^{-10}$. Figure \ref{fig:svm} plots the gap against regularizer coefficient $\lambda$; we consider $\lambda \in \{10^{-4},10^{-3}, 10^{-2}, 10^{-1}, 1.0, 2.0\}$.

In this simple setting, the parallel bundle method out-of-the-box performs similarly to the tuned subgradient method without the need of function related information. We see that the parallel method can handle a wide range of parameters $\lambda$ with a scarce set of potential stepsizes. Notice that while for small $\lambda$ the performance of the subgradient method tends to deteriorate, the performance of the bundle method improves.

  \subsection{Log-Sum-Exp function}

  \begin{figure}
	\centering
	\begin{subfigure}[b]{0.3\textwidth}
		\centering
		\includegraphics[width=\textwidth]{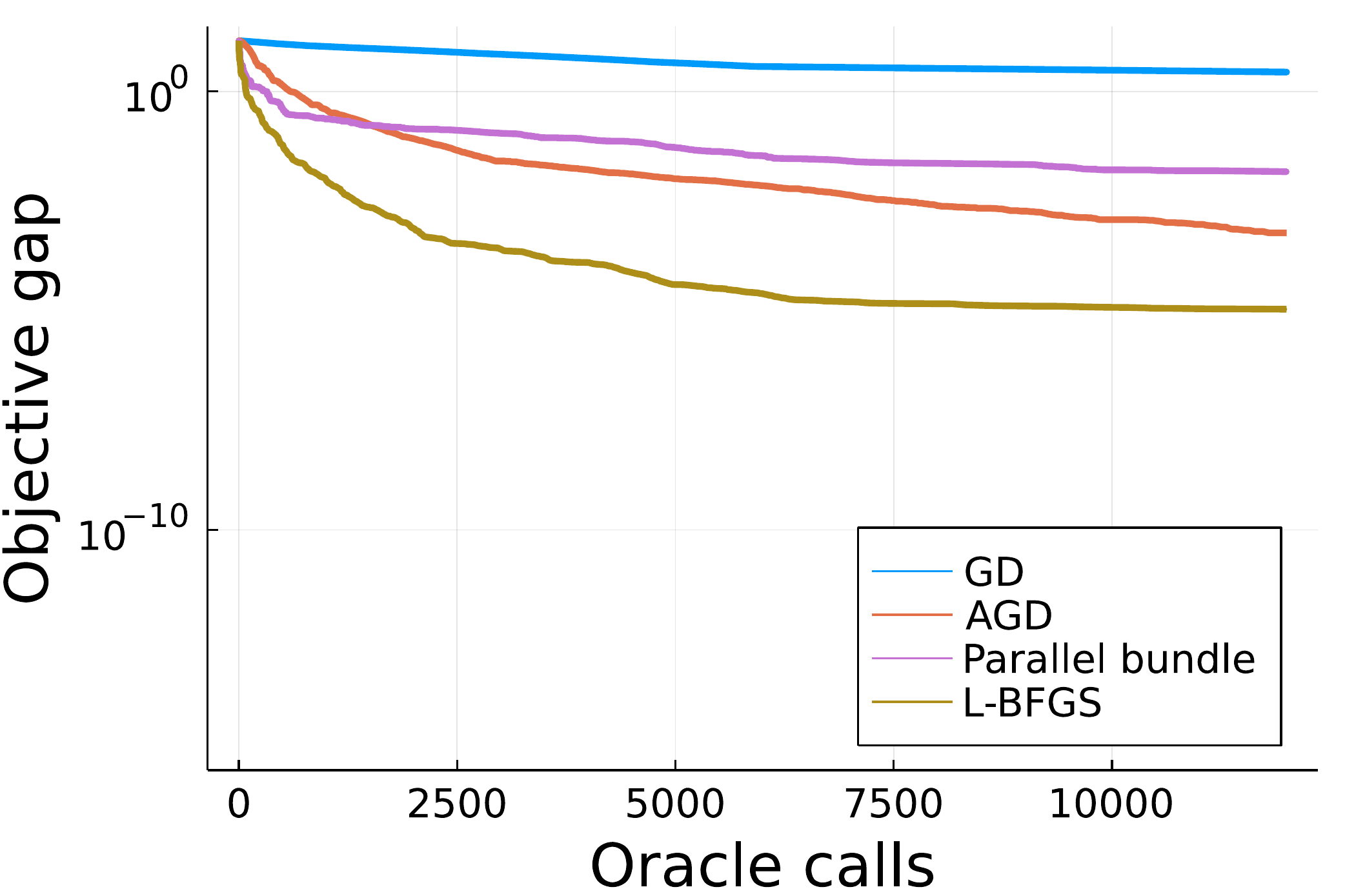}
	\end{subfigure}
	\hfill
	\begin{subfigure}[b]{0.3\textwidth}
		\centering
		\includegraphics[width=\textwidth]{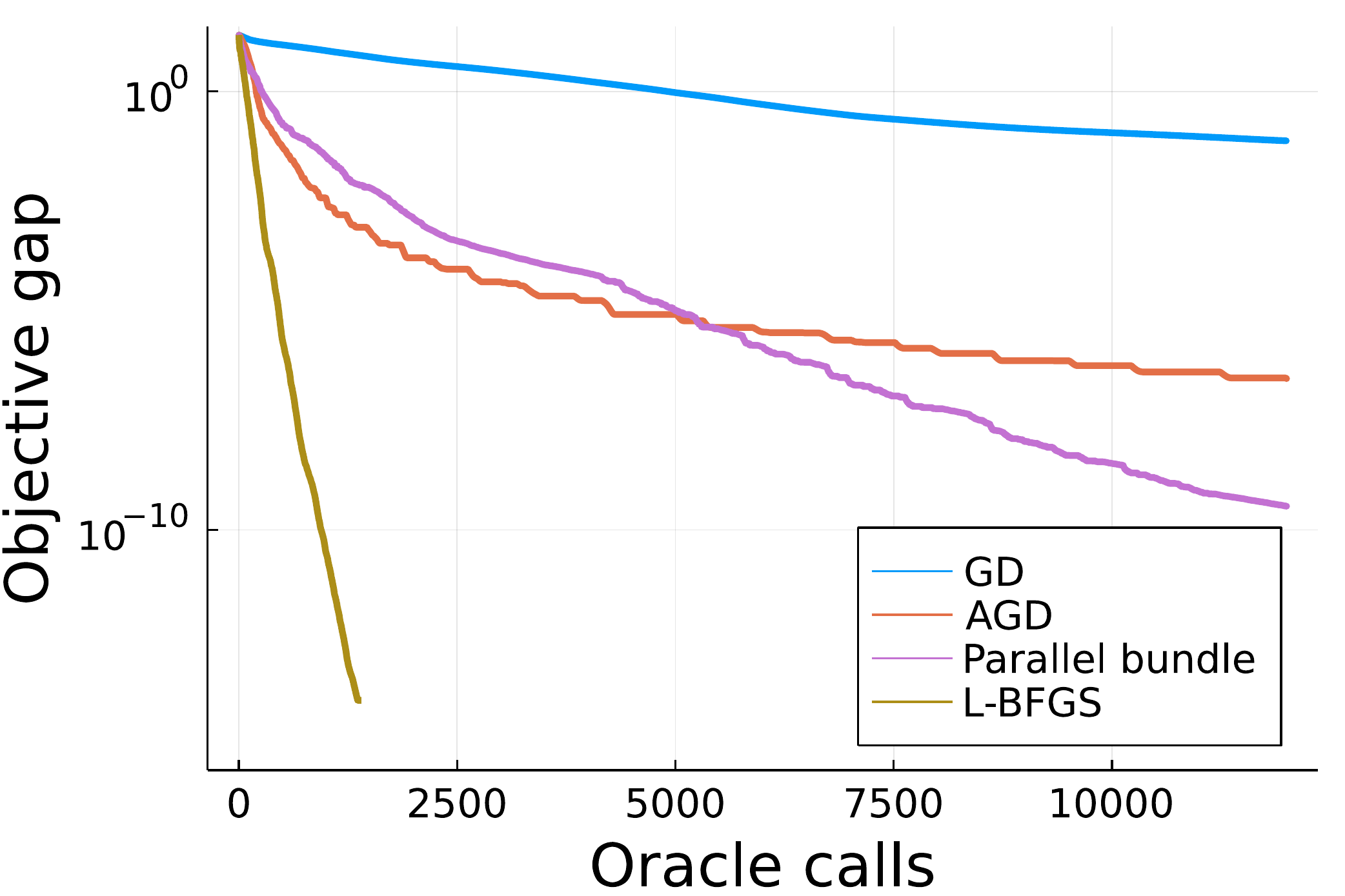}
	\end{subfigure}
	\hfill
	\begin{subfigure}[b]{0.3\textwidth}
		\centering
		\includegraphics[width=\textwidth]{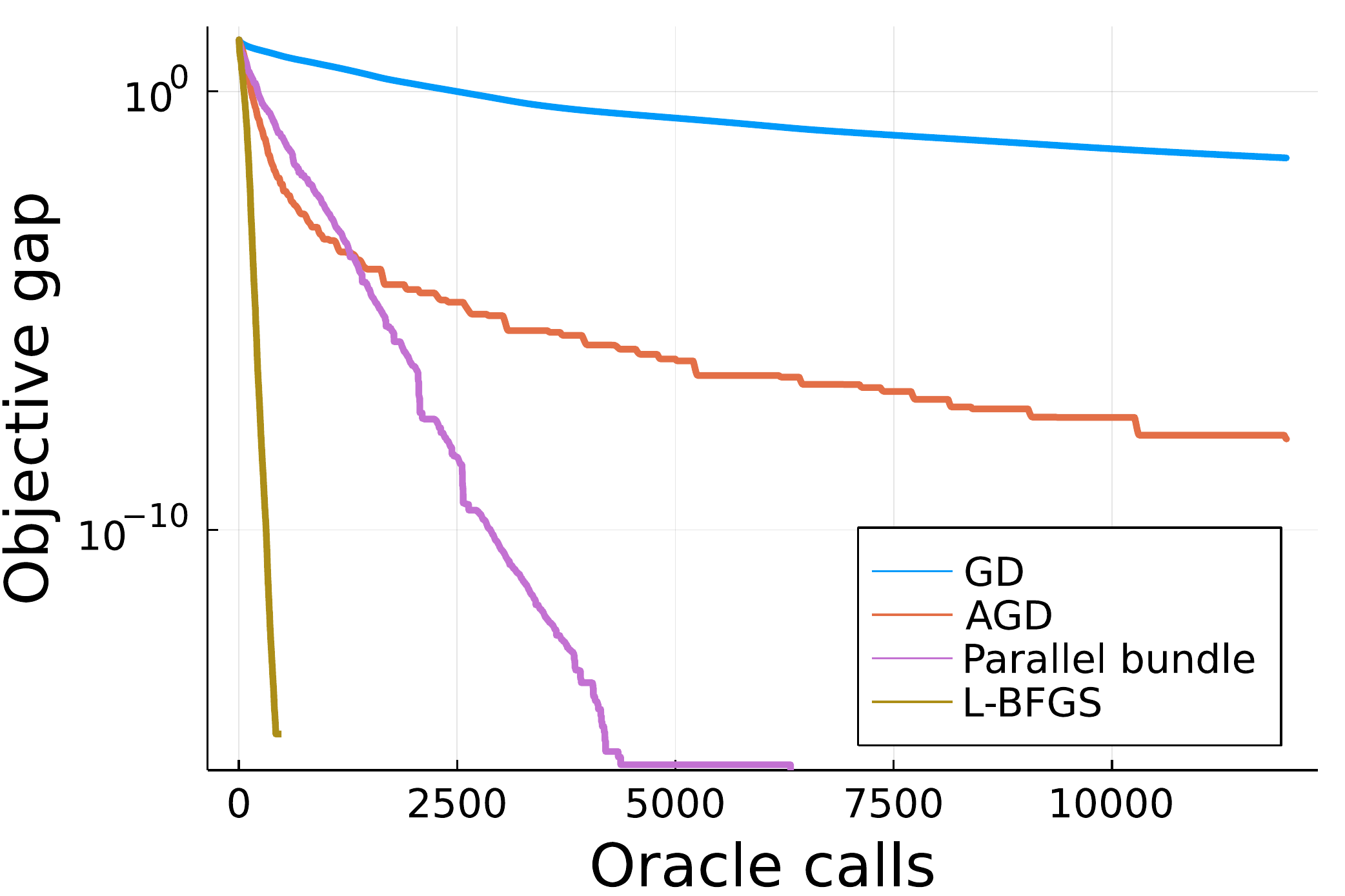}
	\end{subfigure}
	\caption{Best objective gap against number of oracle calls for the Log-Sum-Exp problem. Results for different parameters $\gamma$: $0.01$ (left), $0.05$ (middle), and $0.08$ (right).}
	\label{fig:softmax}
\end{figure}
In this section, we aim to illustrate the performance of the bundle method on smooth problems. We consider a soft-max problem of the form:
$$
\argmin_{x \in \RR^{d}} f(x; A, b) := \gamma \log\left(\sum_{i= 1}^{n}\exp\left(\frac{\dotp{a_{i}, x} - b_{i}}{\gamma}\right)\right).
$$
where $a_{i}$ is the $i$th column of the matrix $A$. This problem lacks strong convexity and becomes ``stiffer'' as $\gamma$ approaches zero. We generate random instances as follows: we let $b_{i}$ be i.i.d. random variables with distribution $\text{Unif} [-1, 1]$, similarly, we draw a random matrix $\widehat A$ with i.i.d. entries $(\widehat a_{i})_{j} \sim \text{Unif} [-1, 1]$ and set $a_{i} = \widehat a_{i} - \nabla f(0; \widehat A, b).$ This choice ensures that zero is a minimizer of $f(\cdot; A, b).$

We fix $d = 100$ and $n = 6 d$ and vary the parameter $\gamma \in \{0.01, 0.05, 0.08\}.$ We test four algorithms: Gradient Descent(GD), Accelerated Gradient Descent (AGD), Limited-Memory Broyden–Fletcher–Goldfarb–Shanno algorithm (L-BFGS) \cite{byrd1994representations}, and the parallel bundle method.  We use the open source implementation of GD, AGD, and L-BFGS provided by \texttt{Optim.jl} \cite{mogensen2018optim}. We set the memory of L-BFGS to ten and use four instances, with stepsizes $\rho \in \{10^{-3}, 10^{{-2}}, 10^{-1}, 1\},$ for the parallel bundle method. This way, L-BFGS and the bundle method have a comparable memory footprint. GD and AGD use the theoretically sound step size $0.9 / L$, where $L$ is the smoothness constant of $f$. While L-BFGS uses backtracking line search \cite[Section 3.5]{NW}.


Figure \ref{fig:softmax} plots the best seen objective gap $f - \min f$ against the number of oracle calls for the different values of $\mu$. L-BFGS is faster than all the other algorithms in all scenarios; this is somewhat expected since L-BFGS is widely recognized for its superior practical performance. {Nonetheless,
the parallel bundle method consistently outperforms GD and is often competitive with AGD, despite its very real practical drawback of utilizing $J=4$ oracle calls per iteration.}
It would be interesting to explore whether one can accelerate the bundle method to make it competitive with AGD when $\mu$ is small. This is left as an intriguing question for future research.

\section{Analysis}\label{sec:theory-proofs}

In this section, we develop the proofs of the convergence rates. We start by introducing the general strategy that we use to establish all of our results and then specialize it to each scenario.
\subsection{Analysis Overview and Proof Sketch}
Each iteration of the bundle method can be viewed as an attempt to mimic the proximal point method, using the model $f_k$ instead of the true objective function $f$. At each iteration $k$, we denote the objective gap of the proximal subproblem, called the {\it proximal gap}, by
$$ \Delta_k := f(x_k) - \left(f(\bar{x}_{k+1}) + \frac{\rho_k}{2}\|\bar{x}_{k+1} - x_k\|^2\right)$$
where $\bar{x}_{k+1} = \argmin_{x\in\RR^d} \left\{f(x) + \frac{\rho_k}{2}\|x-x_k\|^2\right\}$.

Regardless of which continuity, smoothness and growth assumptions are made, our analysis works by relating the proximal steps computed by the bundle method on the models $f_k$ to proximal steps on $f$. The following pair of observations show that the behavior on both descent steps and null steps is controlled by the proximal gap $\Delta_k$.
\begin{itemize}
\item[(i)] {\bf Descent steps attain decrease proportional to the proximal gap.}
  \begin{lemma}\label{lem:descent-prox-gap}
    A descent step, at iteration $k$, has
    \begin{equation} \nonumber
      f(x_{k+1}) \leq f(x_k) - \beta\Delta_k.
    \end{equation}
  \end{lemma}
\item[(ii)] {\bf The number of consecutive null steps is bounded by the proximal gap.}
  \begin{lemma}\label{lem:null-prox-gap}
    A descent step, at iteration $k$, followed by $T$ consecutive null steps has at most
    \begin{equation}\nonumber
      T\leq \frac{8G_{k+1}^2}{(1-\beta)^2\rho_{k+1}\Delta_{k+T}}
    \end{equation}
    where $G_{k+1}=\sup \{\|g_{t+1}\| \mid k\leq t \leq k+T\}$. This simplifies to
    $$ T \leq \begin{cases}
      \dfrac{2M^2}{(1-\beta)^2\rho_{k+1}\Delta_{k+T}} & \text{ if $f$ is $M$-Lipschitz, or}\\
      \dfrac{16(L+\rho_{k+1})^3}{(1-\beta)^2\rho^3_{k+1}} & \text{ if $f$ is $L$-smooth} \ .
    \end{cases} $$
  \end{lemma}
\end{itemize}

With these two observations in hand, convergence guarantees for the bundle method follow from specifying any choice of the parameter $\rho_k$.
For fixed $\rho_k$, bounding the proximal gap is a classic, well-understand problem. Standard analysis~\cite[Lemma 7.12]{Ruszczynski2006} of the proximal gap shows the following bound for any minimizer $x^*$.\footnote{\cite[Lemma 7.12]{Ruszczynski2006} is missing a ``$1/2$'' in its statement, however it does appear in the proof.}
\begin{lemma} \label{lem:prox-gap-bound}
  Fix a minimizer $x^{*}$ of $f$ and let $x_k\in\RR^n \setminus \{x^{*}\}$. Then, proximal gap is lower bounded by
  \begin{equation} \label{eq:prox-gap-bound}
    \Delta_k \geq \begin{cases} \dfrac{1}{2\rho_k}\left(\dfrac{f(x_k)-f(x^*)}{\|x_k-x^*\|}\right)^2 & \text{ if } f(x_k)-f(x^*) \leq \rho_k\|x_k-x^*\|^2\\
      f(x_k)-f(x^*)-\dfrac{\rho_k}{2}\|x_k-x^*\|^2 & \text{ otherwise.} \end{cases}
  \end{equation}
\end{lemma}
\noindent Our ideal stepsize~\eqref{eq:ideal-stepsize} is chosen to balance the two cases of this classic bound, ensuring $\Delta_k\geq \frac{1}{2}(f(x_k)-f(x^*))$. This lemma then gives insight into the effect of approximating this stepsize by some $\rho_k = \lambda (f(x_k)-f(x^*))/\|x_k-x^*\|^2$. Using an over estimate of the target stepsize (i.e., $\lambda >1$), one has $\Delta_k \geq \frac{1}{2\lambda}(f(x_k)-f(x^*))$, weakening the descent bound (Lemma~\ref{lem:descent-prox-gap}) by a factor of $1/\lambda$. Similarly, underestimating the target stepsize (i.e., $\lambda <1$), one maintains $\Delta_k\geq \frac{1}{2}(f(x_k)-f(x^*))$, but the smaller $\rho_k$ leads the null step bound (Lemma~\ref{lem:null-prox-gap}) to grow by $1/\lambda$. Similar reasoning holds for approximating the stepsizes~\eqref{eq:OPTgeneral-stepsize} and \eqref{eq:OPTholder-stepsize}.

All of our analysis 
follows directly from applying these core lemmas. We bound the number of descent steps by combining Lemmas~\ref{lem:descent-prox-gap} and~\ref{lem:prox-gap-bound} to give a recurrence relation describing the decrease in the objective gap. Then Lemmas~\ref{lem:null-prox-gap} and~\ref{lem:prox-gap-bound} together allow us to bound the number of consecutive null steps between each of these descent steps, which can then be summed up to bound the total number of iterations required.

%
%
%


\subsection{Proof of the Descent Step Lemma~\ref{lem:descent-prox-gap}}
Let $\bar{x}_{k+1} = \argmin\{f(\cdot) + \frac{\rho_k}{2}\|\cdot-x_k\|^2\}$. From~\eqref{eq:model-lowerBound}, we have
\begin{align*}
  f_k(x_{k+1}) &\leq f_k(x_{k+1})+\frac{\rho_k}{2}\|x_{k+1}-x_{k}\|^2\\
               &\leq f_k(\bar{x}_{k+1})+\frac{\rho_k}{2}\|\bar{x}_{k+1}-x_{k}\|^2\\
               &\leq f(\bar{x}_{k+1})+\frac{\rho_k}{2}\|\bar{x}_{k+1}-x_{k}\|^2 \ .
\end{align*}
The second line follows by the definition of $x_{k+1}$ and the last line uses $f_{k} \leq f.$ Hence $f(x_k) - f_k(x_{k+1}) \geq \Delta_k$. Since we have assumed that iteration $k$ was a descent step, this implies $(f(x_k) - f(x_{k+1}))/\beta \geq \Delta_k$. Concluding the proof.

\subsection{Proof of the Null Step Lemma~\ref{lem:null-prox-gap}}
Consider some descent step, at iteration $k$, followed by $T$ consecutive null steps.
Denote the proximal subproblem gap at iteration $k< t\leq k+T$ on the model $f_t$ by
$$\widetilde\Delta_t := f(x_{k+1}) - \left(f_t(z_{t+1}) + \frac{\rho_{t}}{2}\|z_{t+1} - x_{k+1}\|^2\right).$$
Note every such null step $t$ has stepsize $\rho_t\geq \rho_{k+1}$ and the same proximal center $x_t=x_{k+1}$.
The core of this null step bound relies on the following recurrence showing $\widetilde\Delta_t$ decreases at each step
\begin{equation} \label{eq:null-recurrence}
  \widetilde\Delta_{t+1} \leq \widetilde\Delta_{t} -  \frac{(1-\beta)^2\rho_{k+1}\widetilde\Delta_t^2}{8G_{k+1}^2} \ .
\end{equation}
Before deriving this inequality, we show how it completes the proof of this lemma.
After $T$ consecutive null steps, the fact that $f_{k+T}\leq f$ ensures $\widetilde \Delta_{k+T} \geq \Delta_{k+T}$. Thus, to bound $T$ it suffices to bound the minimum iteration at which the reversed inequality hold. By solving the recurrence~\eqref{eq:null-recurrence}, see Lemma~\ref{lem:recurrence} in the appendix with $\epsilon = \Delta_{k+T}$, we conclude the number of consecutive null steps is at most
$$ T\leq \frac{8G_{k+1}^2}{(1-\beta)^2\rho_{k+1}\Delta_{k+T}} \ . $$

Now all that remains is to derive the recurrence~\eqref{eq:null-recurrence}. Consider some null step $k< t\leq k+T$ in the sequence of consecutive null steps. It suffices to assume $\rho_{t+1}=\rho_t$ as increasing the value of the proximal parameter can only further decrease the proximal gap.
Define the necessary lower bound on $f_{t+1}$ given by~\eqref{eq:model-subgradf} and~\eqref{eq:model-subgradfk} as
$$\tilde f_{t+1}(\cdot) := \max\left\{ f_{t}(z_{t+1}) + \langle s_{t+1}, \cdot-z_{t+1}\rangle,\ f(z_{t+1}) + \langle g_{t+1}, \cdot-z_{t+1}\rangle\right\} \leq f_{t+1}(\cdot)\ .$$
Denote the result of a proximal step on $\tilde f_{t+1}$ by $y_{t+2} = \argmin\left\{\tilde f_{t+1}(\cdot) + \frac{\rho_{t}}{2}\|\cdot-x_{k+1}\|^2\right\}\ .$
	\begin{claim}\label{claim:closedForm} The solution to the above optimization problem is:
		\begin{align}
		\theta_{t+1} &= \min\left\{1,\frac{\rho_{t}\left(f(z_{t+1})-f_t(z_{t+1})\right)}{\|g_{t+1}-s_{t+1}\|^2}\right\}\nonumber\\
		y_{t+2} &= x_{k+1} - \frac{1}{\rho_{t}}\left(\theta_{t+1}g_{t+1} +(1-\theta_{t+1})s_{t+1}\right).\label{eq:two-cut-solution}
		\end{align}
	\end{claim}
  This claim is a well-known fact, we include a brief proof in Appendix~\ref{app:proofClaim} for completion.

	Considering $y_{t+2}$, the objective of the proximal subproblem at iteration $t+1$ is lower bounded by
	\begin{align*}
	&f_{t+1}(z_{t+2}) + \frac{\rho_{t}}{2}\|z_{t+2} - x_{k+1}\|^2\\ &\geq \tilde f_{t+1}(y_{t+2}) + \frac{\rho_{t}}{2}\|y_{t+2} - x_{k+1}\|^2\\
	&\geq \theta_{t+1}\left(f(z_{t+1}) + \langle g_{t+1}, y_{t+2}-z_{t+1}\rangle\right)  \\
	& \hspace{1cm}+ (1-\theta_{t+1})\left(f_{t}(z_{t+1}) + \langle s_{t+1}, y_{t+2}-z_{t+1}\rangle\right) + \frac{\rho_{t}}{2}\|y_{t+2} - x_{k+1}\|^2\\
	&= f_{t}(z_{t+1}) + \theta_{t+1}\left(f(z_{t+1}) - f_{t}(z_{t+1})\right) \\
	& \hspace{1cm}+ \langle \theta_{t+1}g_{t+1} + (1-\theta_{t+1})s_{t+1}, y_{t+2}-z_{t+1}\rangle + \frac{\rho_{t}}{2}\|y_{t+2} - x_{k+1}\|^2\\
	&= f_{t}(z_{t+1}) + \theta_{t+1}\left(f(z_{t+1}) - f_{t}(z_{t+1})\right) \\
	& \hspace{1cm}- \theta_{t+1}^2\|g_{t+1}-s_{t+1}\|^2/2\rho_{t} + \frac{\rho_{t}}{2}\|z_{t+1} - x_{k+1}\|^2 \ ,
	\end{align*}
	where the first inequality uses that $f_{t+1}\geq \tilde f_{t+1}$, the second inequality takes a convex combination of the two affine functions defining $\tilde f_{t+1}$, and the second equality uses the definition of $y_{t+2}$ to complete the square (noting $y_{k+2}-z_{k+1}=\theta_{t+1}/\rho_{t}(s_{t+1}-g_{t+1})$).
	Thus we have
	$$ \widetilde\Delta_{t+1} \leq \widetilde\Delta_{t} - \theta_{t+1}\left(f(z_{t+1}) - f_{t}(z_{t+1})\right) + \theta_{t+1}^2\|g_{t+1}-s_{t+1}\|^2/2\rho_{t} \ .$$
	The amount of decrease guaranteed above can be lower bounded as follows
	\begin{align*}
	&\theta_{t+1}\left(f(z_{t+1}) - f_{t}(z_{t+1})\right) - \theta_{t+1}^2\|g_{t+1}-s_{t+1}\|^2/2\rho_{t}\\
	&\geq \frac{1}{2}\min\left\{f(z_{t+1}) - f_{t}(z_{t+1}),\ \frac{\rho_{t}(f(z_{t+1})-f_t(z_{t+1}))^2}{\|g_{t+1}-s_{t+1}\|^2}\right\}\\
	&\geq \frac{1}{2}\min\left\{(1-\beta)\widetilde\Delta_t,\ \frac{\rho_{t}(1-\beta)^2\widetilde\Delta_t^2}{\|g_{t+1}-s_{t+1}\|^2}\right\}\\
	&\geq \frac{1}{2}\min\left\{(1-\beta)\widetilde\Delta_t,\ \frac{\rho_{t}(1-\beta)^2\widetilde\Delta_t^2}{2\|g_{t+1}\|^2+2\|s_{t+1}\|^2}\right\}
	\end{align*}
	where the first inequality uses the definition of $\theta_{t+1}$, the second inequality uses the definition of a null step.
	Noting that both components of this minimum above are nonnegative, we conclude the weaker results that $\widetilde\Delta_{t+1}$ is nonincreasing. Utilizing the chain of inequalities $\|s_{t+1}\|^2 \leq 2 \rho_{t} \tilde \Delta_t \leq (\rho_t/\rho_{k+1})G_{k+1}^2$ (which we will show in a second) and that $\rho_t\geq\rho_{k+1}$, the decrease bound above is at least
	\begin{align*}
	&\geq \frac{1}{2}\min\left\{(1-\beta)\widetilde\Delta_t,\ \frac{\rho_{t}(1-\beta)^2\widetilde\Delta_t^2}{2\|g_{t+1}\|^2+2\|s_{t+1}\|^2}\right\}\\
	&\geq \frac{1}{2}\min\left\{2\frac{\rho_{k+1}(1-\beta) \widetilde\Delta_t^2}{G_{k+1}^2},\ \frac{\rho_{t}(1-\beta)^2\widetilde\Delta_t^2}{4G_{k+1}^2}\right\}\\
	&\geq  \frac{\rho_{k+1}(1-\beta)^2\widetilde\Delta_t^2}{8G_{k+1}^2} \ .
	\end{align*}
	This needed chain of inequalities holds as
	\begin{align}\begin{split}\label{eq:grad_bound}
	\frac{\rho_{t}}{2}\|z_{t+1} - x_{k+1} \|^2 &\leq f_t(x_{k+1}) - \left(f_t(z_{t+1}) + \frac{\rho_{t}}{2}\|z_{t+1} - x_{k+1}\|^2\right)\\
	&\leq \widetilde \Delta_t\\
	&\leq \widetilde \Delta_{k+1} \\
	&\leq \frac{1}{2\rho_{k+1}}\|g_{k+1}\|^2
	\end{split}
	\end{align}
	where the first inequality uses the  $\rho_{t}$-strongly convexity of the proximal subproblem $f_{t}(z)+\frac{\rho_{t}}{2}\|z-x_{k+1}\|^2$, the second uses the fact that $\tilde \Delta_{t}$ is nonincreasing, and the last inequality follows by~\eqref{eq:model-subgradf} applied at the descent step (i.e., $x_{k+1}=z_{k+1}$).

For any $M$-Lipschitz objective, our specialized result follows from observing that $G_k\leq M$ as subgradients everywhere are uniformly bounded in norm by the Lipschitz constant.
For any $L$-smooth objective, the following three inequalities hold for any null step $t$ in the sequence of consecutive null steps following a descent step $k<t$:
\begin{align}
  \|g_{t+1}\| & \leq \|g_{k+1}\| + L\|z_{t+1}-x_{k+1}\| \label{eq:Gk-bound-eq1}\\
  \|z_{t+1}-x_{k+1}\| & \leq \|g_{k+1}\|/\rho_{k+1} \label{eq:Gk-bound-eq2}\\
  \|g_{k+1}\| & \leq \sqrt{2(L+\rho_{k+1})\Delta_{k+1}} \ . \label{eq:Gk-bound-eq3}
\end{align}
Combined these three inequalities give the claimed bound as
\begin{align*}
  G_{k+1} = \sup_t\{\|g_{t+1}\|\}&\leq \sup_t\{\|g_{k+1}\| + L\|z_{t+1}-x_{k+1}\|\} \\
                                 &\leq (1+L/\rho_{k+1})\|g_{k+1}\|\\
                                 &\leq (1+L/\rho_{k+1})\sqrt{2(L+\rho_{k+1})\Delta_{k+1}}
\end{align*}
and thus $G_{k+1}^2 \leq 2(L+\rho_{k+1})^3\Delta_{k+1}/\rho_{k+1}^2$.
First~\eqref{eq:Gk-bound-eq1} follows directly from the gradient being $L$-Lipschitz continuous. Second~\eqref{eq:Gk-bound-eq2} follows from \eqref{eq:grad_bound}.
Third~\eqref{eq:Gk-bound-eq3} follows from the $L$-smoothness of $f$ and considering the full proximal subproblem $f(z)+\frac{\rho_{k+1}}{2}\|z-x_{k+1}\|^2$ since
\begin{align*}
  \Delta_{k+1} &= f(x_{k+1}) - \min_z\left\{ f(z) + \frac{\rho_{k+1}}{2}\|z-x_{k+1}\|^2 \right\} \\
               &\geq f(x_{k+1}) - \min_z \left\{ f(x_{k+1})+\langle g_{k+1}, z - x_{k+1}\rangle + \frac{L+\rho_{k+1}}{2}\|z-x_{k+1}\|^2\right\} \\
               &= \frac{\|g_{k+1}\|^2}{2(L+\rho_{k+1})} \ .
\end{align*}

\subsection{Proof of Theorem~\ref{thm:LipschitzGeneral}}\label{proof:LipschitzGeneral}

For a constant stepsize $\rho_k=\rho$, we can simplify the lower bound~\eqref{eq:prox-gap-bound} to only depend on $x_k$ through a simple threshold on $f(x_k)-f^*$ as
\begin{equation} \label{eq:prox-gap-bound-general}
  \Delta_k \geq \begin{cases} \dfrac{1}{2\rho}\left(\dfrac{f(x_k)-f^*}{D}\right)^2 & \text{ if } f(x_k)-f^* \leq \rho D^2\\
    \dfrac{1}{2}\left(f(x_k)-f^*\right) & \text{ otherwise. } \end{cases}
\end{equation}
Combining this with Lemma~\ref{lem:descent-prox-gap} gives a recurrence relation describing the decrease in the objective gap $\delta_k=f(x_k) - f^*$ on any descent step $k$ of
$$ \delta_{k+1} \leq
\begin{cases}
  \delta_k - \dfrac{\beta\delta_k^2}{2\rho D^2} & \text{ if } \delta_k \leq \rho D^2 \\
  (1-\beta/2)\delta_k & \text{ if } \delta_k > \rho D^2\ .
\end{cases}$$
Our analysis of the bundle method then proceeds by considering these two cases separately. In each case, solving the given recurrence relation bounds the number of descent steps and applying Lemma~\ref{lem:null-prox-gap} bounds the number of null steps.

\subsubsection{Bounding steps with $\delta_k > \rho D^2$.}
First we show that the number of descent steps with $\delta_k > \rho D^2$ is bounded by
\begin{equation}\label{eq:LipschitzGeneral-descent-case1}
  \left\lceil \frac{\log\left(\frac{f(x_0)-f^*}{\rho D^2}\right)}{-\log(1-\beta/2)} \right\rceil_+
\end{equation}
and the number of null steps with $\delta_k > \rho D^2$ is at most
\begin{equation}\label{eq:LipschitzGeneral-null-case1}
  \frac{32M^2}{\beta(1-\beta)^2\rho^2D^2} \ .
\end{equation}

In this case, our recurrence relation simplifies to have geometric decrease at each descent step $\delta_{k+1}\leq (1-\beta/2)\delta_k$. This immediately bounds the number of descent steps by~\eqref{eq:LipschitzGeneral-descent-case1}. Index the descent steps before a $\rho D^2$-minimizer is found by $k_1<\dots< k_n$ such that $x_{k_n+1}$ is the first iterate with objective value less than $\rho D^2$. Define $k_0=-1$. Then for each $i=0\dots n-1$,
$ f(x_{k_i+1}) - f^* \geq (1-\beta/2)^{i-(n-1)}\rho D^2 \ .$
It follows from~\eqref{eq:prox-gap-bound} that
$\Delta_{k_i+1} \geq (f(x_{k_i+1})-f^*)/2 \geq \left(1-\beta/2\right)^{i-(n-1)}\rho D^2/2.$
Plugging this into Lemma~\ref{lem:null-prox-gap} (note $\Delta_{k+T}=\Delta_{k+1}$ since the stepsize is held constant) upper bounds the number of consecutive null steps after the descent step $k_i$ by
$$ k_{i+1}-k_{i}-1 \leq \left(1-\beta/2\right)^{(n-1)-i}\frac{16M^2}{(1-\beta)^2\rho^2D^2} \ .$$
Summing this over $i=0\dots n-1$ bounds the total number of null steps before a $\rho D^2$-minimizer is found by~\eqref{eq:LipschitzGeneral-null-case1} as
$$ \sum_{i=0}^{n-1} \left(1-\beta/2\right)^{(n-1)-i}\frac{16M^2}{(1-\beta)^2\rho^2D^2} \leq \frac{32M^2}{\beta(1-\beta)^2\rho^2D^2} \ .$$

\subsubsection{Bounding steps with $\rho D^2 \geq \delta_k > \epsilon$.}
Now we complete our proof of Theorem~\ref{thm:LipschitzGeneral} by bounding the number of descent steps with $\rho D^2 \geq \delta_k > \epsilon$ by
\begin{equation}\label{eq:LipschitzGeneral-descent-case2}
  \frac{2\rho D^2}{\beta\epsilon}
\end{equation}
and the number of null steps with $\rho D^2 \geq \delta_k > \epsilon$ by
\begin{equation}\label{eq:LipschitzGeneral-null-case2}
  \frac{48\rho D^4M^2}{(1-\beta)^2\epsilon^3} \ .
\end{equation}

After the bundle method has passed objective value $\rho D^2$, the recurrence relation becomes
$$ \delta_{k+1} \leq \delta_k - \dfrac{\beta\delta_k^2}{2\rho D^2}. $$
Solving this recurrence with Lemma~\ref{lem:recurrence} implies $\delta_k > \epsilon$ holds for at most~\eqref{eq:LipschitzGeneral-descent-case2} descent steps.  Then we can bound the number of null steps between these descent steps by noting~\eqref{eq:prox-gap-bound-general} implies
$\Delta_k \geq (f(x_k)-f^*)^2/2\rho D^2 \geq \epsilon^2/2\rho D^2.$
Then Lemma~\ref{lem:null-prox-gap} (note $\Delta_{k+T}=\Delta_{k+1}$ since the stepsize is held constant) upper bounds the number of consecutive null steps by
$ 16D^2M^2/(1-\beta)^2\epsilon^2. $
Then multiplying this by our bound on the number of descent steps gives~\eqref{eq:LipschitzGeneral-null-case2} as
$$ \left(\frac{2\rho D^2}{\beta\epsilon} +1 \right) \frac{16D^2M^2}{(1-\beta)^2\epsilon^2}\leq \frac{48\rho D^4M^2}{\beta(1-\beta)^2\epsilon^3} \ . $$

\subsection{Proof of Theorem~\ref{thm:SmoothGeneral}}\label{proof:SmoothGeneral}
Our bound on the number of descent steps comes directly from Theorem~\ref{thm:LipschitzGeneral}.
Our claimed bound on the total number of null steps follows by multiplying this by the constant bound on the number of consecutive null steps from Lemma~\ref{lem:null-prox-gap}.

\subsection{Proof of Theorem~\ref{thm:LipschitzGrowth}} \label{proof:LipschitzGrowth}
Assuming H\"older growth~\eqref{eq:holder-growth} holds and fixing $\rho_k=\rho$, the lower bound~\eqref{eq:prox-gap-bound} simplifies to only depend on a simple threshold with $f(x_k)-f^*$ as
\begin{equation} \label{eq:prox-gap-bound-growth}
  \Delta_k \geq \begin{cases} \dfrac{\mu^{2/p}(f(x_k)-f^*)^{2-2/p}}{2\rho} & \text{ if } (f(x_k)-f^*)^{1-2/p} \leq \rho/\mu^{2/p}  \\
    \dfrac{1}{2}\left(f(x_k)-f^*\right) & \text{ otherwise .} \end{cases}
\end{equation}
From this, we arrive at a recurrence relation on the objective gap $\delta_k=f(x_k) - f^*$ decrease at each descent step $k$ by plugging this lower bound into Lemma~\ref{lem:descent-prox-gap} of
$$ \delta_{k+1} \leq
\begin{cases}
  \delta_k - \dfrac{\beta\mu^{2/p}\delta_k^{2-2/p}}{2\rho} & \text{ if } \delta_k^{1-2/p} \leq \rho/\mu^{2/p}\\
  (1-\beta/2)\delta_k & \text{ if } \delta_k^{1-2/p} > \rho/\mu^{2/p}\ .
\end{cases}$$
Our analysis proceeds by considering the two cases of this recurrence and the three cases of $p>2$, $p=2$, and $1\leq p <2$ separately. In each case, solving the given recurrence relation bounds the number of descent steps and applying Lemma~\ref{lem:null-prox-gap} bounds the number of null steps.

\subsubsection{Given $p>2$, bounding steps with  $\delta_k > (\rho/\mu^{2/p})^{1/(1-2/p)}$.}
First we show that the number of descent steps with $\delta_k > (\rho/\mu^{2/p})^{1/(1-2/p)}$ is bounded by
\begin{equation}\label{eq:LipschitzGrowth-descent-case1}
  \left\lceil \frac{\log\left(\frac{f(x_0)-f^*}{(\rho/\mu^{2/p})^{1/(1-2/p)}}\right)}{-\log(1-\beta/2)} \right\rceil_+
\end{equation}
and the number of null steps with $\delta_k > (\rho/\mu^{2/p})^{1/(1-2/p)}$ is at most
\begin{equation}\label{eq:LipschitzGrowth-null-case1}
  \frac{8M^2}{\beta(1-\beta)^2\rho(\rho/\mu^{2/p})^{1/(1-2/p)}} \ .
\end{equation}

In this case, our recurrence relation simplifies to have geometric decrease at each descent step $\delta_{k+1}\leq (1-\beta/2)\delta_k$. This immediately bounds the number of descent steps by~\eqref{eq:LipschitzGrowth-descent-case1}. Index the descent steps before a $(\rho/\mu^{2/p})^{1/(1-2/p)}$-minimizer is found by $k_1<\dots< k_n$ such that $x_{k_n+1}$ is the first iterate with objective value less than $(\rho/\mu^{2/p})^{1/(1-2/p)}$. Define $k_0=-1$. Then for each $i=0\dots n-1$,
$ f(x_{k_i+1}) - f^* \geq (1-\beta/2)^{i-(n-1)}(\rho/\mu^{2/p})^{1/(1-2/p)}.$
It follows from~\eqref{eq:prox-gap-bound} that
$$\Delta_{k_i+1} \geq (f(x_{k_i+1})-f^*)/2 \geq \left(1-\beta/2\right)^{i-(n-1)}(\rho/\mu^{2/p})^{1/(1-2/p)}/2.$$
Plugging this into Lemma~\ref{lem:null-prox-gap} (note $\Delta_{k+T}=\Delta_{k+1}$ since the stepsize is held constant) upper bounds the number of consecutive null steps after the descent step $k_i$ by
$$ k_{i+1}-k_{i}-1 \leq \left(1-\beta/2\right)^{(n-1)-i}\frac{16M^2}{(1-\beta)^2\rho(\rho/\mu^{2/p})^{1/(1-2/p)}} \ .$$
Summing this over $i=0\dots n-1$ bounds the total number of null steps before a $(\rho/\mu^{2/p})^{1/(1-2/p)}$-minimizer is found by~\eqref{eq:LipschitzGrowth-null-case1} as
$$ \sum_{i=0}^{n-1} \left(1-\beta/2\right)^{(n-1)-i}\frac{16M^2}{(1-\beta)^2\rho(\rho/\mu^{2/p})^{1/(1-2/p)}} \leq \frac{32M^2}{\beta(1-\beta)^2\rho(\rho/\mu^{2/p})^{1/(1-2/p)}} \ .$$

\subsubsection{Given $p>2$, bounding steps with  $(\rho/\mu^{2/p})^{1/(1-2/p)}\geq \delta_k >\epsilon$.}
Next we show that the total number of descent steps with$(\rho/\mu^{2/p})^{1/(1-2/p)}\geq \delta_k >\epsilon$ is bounded by
\begin{equation}\label{eq:LipschitzGrowth-descent-case2}
  \frac{2\rho}{(1-2/p)\beta\mu^{2/p}\epsilon^{1-2/p}}
\end{equation}
and the number of null steps with $(\rho/\mu^{2/p})^{1/(1-2/p)}\geq \delta_k >\epsilon$ is at most
\begin{equation}\label{eq:LipschitzGrowth-null-case2}
  \frac{48\rho M^2}{(1-2/p)\beta(1-\beta)^2\mu^{4/p}\epsilon^{3-4/p}} \ .
\end{equation}

In this case, the recurrence relation on objective value decrease becomes
$$ \delta_{k+1} \leq \delta_k - \dfrac{\beta\mu^{2/p}\delta_k^{2-2/p}}{2\rho}. $$
Applying Lemma~\ref{lem:recurrence} gives our bound on the number of descent steps with $\delta_k>\epsilon$ in~\eqref{eq:LipschitzGrowth-descent-case2}.
Plugging the lower bound $\Delta_k \geq \mu^{2/p}(f(x_k)-f^*)^{2-2/p}/2\rho \geq \mu^{2/p}\epsilon^{2-2/p}/2\rho$ into Lemma~\ref{lem:null-prox-gap}, the number of consecutive null steps after a descent step is at most
$$ \frac{16M^2}{(1-\beta)^2\mu^{2/p}\epsilon^{2-2/p}} \ . $$
Then multiplying our limit on consecutive null steps by the number of descent steps between finding a $(\rho/\mu^{2/p})^{1/(1-2/p)}$-minimizer and finding an $\epsilon$-minimizer gives the bound~\eqref{eq:LipschitzGrowth-null-case2} as
$$ \left(\frac{2\rho}{(1-2/p)\beta\mu^{2/p}\epsilon^{1-2/p}} + 1\right)\frac{16M^2}{(1-\beta)^2\mu^{2/p}\epsilon^{2-2/p}} \leq \frac{48\rho M^2}{(1-2/p)\beta(1-\beta)^2\mu^{4/p}\epsilon^{3-4/p}} \ .$$

\subsubsection{Given $p=2$, bounding steps with $\delta_k >\epsilon$.}
Here both cases of our recurrence relation have a similar form, and so we directly bound the total number of descent steps with $\delta_k > \epsilon$ by
\begin{equation}\label{eq:LipschitzGrowth-descent-case3}
  \left\lceil \frac{\log\left(\frac{f(x_0)-f^*}{\epsilon}\right)}{-\log(1-\beta\min\{\mu/2\rho, 1/2\})} \right\rceil
\end{equation}
and the number of null steps with $\delta_k > \epsilon$ by
\begin{equation}\label{eq:LipschitzGrowth-null-case3}
  \frac{8M^2}{\beta(1-\beta)^2\min\{\mu/2\rho, 1/2\}\rho\epsilon} \ .
\end{equation}

In this case, our recurrence relation simplifies to have geometric decrease at each descent step $\delta_{k+1}\leq (1-\beta\min\{\mu/2\rho, 1/2\})\delta_k$. This immediately bounds the number of descent steps by~\eqref{eq:LipschitzGrowth-descent-case3}.  Index the descent steps before an $\epsilon$-minimizer is found by $k_1<\dots< k_n$ such that $x_{k_n+1}$ is the first iterate with objective value less than $\epsilon$. Define $k_0=-1$. Then for each $i=0\dots n-1$,
$ f(x_{k_i+1}) - f^* \geq (1-\beta\min\{\mu/2\rho, 1/2\})^{i-(n-1)}\epsilon.$
It follows from~\eqref{eq:prox-gap-bound} that
$\Delta_{k_i+1} \geq \left(1-\beta\min\{\mu/2\rho, 1/2\}\right)^{i-(n-1)}\epsilon/2.$
Plugging this into Lemma~\ref{lem:null-prox-gap} (note $\Delta_{k+T}=\Delta_{k+1}$ since the stepsize is held constant) upper bounds the number of consecutive null steps after the descent step $k_i$ by
$$ k_{i+1}-k_{i}-1 \leq \left(1-\beta\min\{\mu/2\rho, 1/2\}\right)^{(n-1)-i}\frac{8M^2}{(1-\beta)^2\rho\epsilon} \ .$$
Summing this over $i=0\dots n-1$ bounds the total number of null steps before an $\epsilon$-minimizer is found by
$$\sum_{i=0}^{n-1} \left(1-\beta\min\{\mu/2\rho, 1/2\}\right)^{(n-1)-i}\frac{8M^2}{(1-\beta)^2\rho\epsilon}\leq \frac{8M^2}{\min\{\mu/2\rho, 1/2\}\beta(1-\beta)^2\rho\epsilon} \ .$$

\subsubsection{Given $1\leq p<2$, bounding steps with  $\delta_k > (\rho/\mu^{2/p})^{1/(1-2/p)}$.}
Now we show that the number of descent steps with $\delta_k > (\rho/\mu^{2/p})^{1/(1-2/p)}$ is bounded by
\begin{equation}\label{eq:LipschitzGrowth-descent-case4}
  \frac{2\rho(f(x_0)-f^*)^{2/p-1}}{(1-2^{1-2/p})\beta\mu^{2/p}}
\end{equation}
and the number of null steps with $\delta_k > (\rho/\mu^{2/p})^{1/(1-2/p)}$ is at most
\begin{equation} \label{eq:LipschitzGrowth-null-case4}
  \frac{32 M^2}{\beta(1-\beta)^2\rho(\rho/\mu^{2/p})^{1/(1-2/p)}}C
\end{equation}
with $C=\max\left\{\frac{(f(x_0)-f^*)^{4/p-3}}{(\rho/\mu^{2/p})^{(4/p-3)/(1-2/p)}}, 1\right\}\min\left\{\frac{1}{1-2^{-|4/p-3|}}, \left\lceil\log_2\left(\frac{f(x_0)-f^*}{(\rho/\mu^{2/p})^{1/(1-2/p)}}\right)\right\rceil\right\}$.
Notice that since $p<2$, the power $1-2/p$ of $\delta_k$ in the threshold condition of our recurrence is negative. Then, the recurrence relation on objective value decrease becomes
$ \delta_{k+1} \leq \delta_k - \dfrac{\beta\mu^{2/p}\delta_k^{2-2/p}}{2\rho}\ . $
For any $i\geq 0$, we first bound the number of descent and null steps with
$$ 2^{i+1}(\rho/\mu^{2/p})^{1/(1-2/p)}\geq \delta_k > 2^i(\rho/\mu^{2/p})^{1/(1-2/p)} \ .$$
Since descent steps decreases the gap by at least $\beta\mu^{2/p}\delta_k^{2-2/p}/2\rho$, there are at most
$$ \frac{2\rho(2^{i}(\rho/\mu^{2/p})^{1/(1-2/p)})^{2/p-1}}{\beta\mu^{2/p}} = \frac{2^{(2/p-1)i+1}}{\beta}$$
descent steps in this interval. Further, noting that in this interval
$$\Delta_{k} \geq \frac{\mu^{2/p}(2^i(\rho/\mu^{2/p})^{1/(1-2/p)})^{2-2/p}}{2\rho} = 2^{(2-2/p)i-1}(\rho/\mu^{2/p})^{1/(1-2/p)} \ ,$$
we can bound the number of consecutive null steps following any of these descent steps via Lemma~\ref{lem:null-prox-gap}. Hence there are at most
$$ \frac{2^{(4/p-3)i+5}M^2}{\beta(1-\beta)^2\rho(\rho/\mu^{2/p})^{1/(1-2/p)}} $$
null steps in this interval.

The bundle method halves its objective value at most $N=\lceil\log_2((f(x_0)-f^*)/(\rho/\mu^{2/p})^{1/(1-2/p)})\rceil$ times before an $(\rho/\mu^{2/p})^{1/(1-2/p)}$-minimizer is found. Then summing up these bounds on the descent and null steps in each interval limits the number of descent steps needed to find a $(\rho/\mu^{2/p})^{1/(1-2/p)}$-minimizer by~\eqref{eq:LipschitzGrowth-descent-case4} as
\begin{align*}
  \sum_{i=0}^{N-1} \frac{2^{(2/p-1)i+1}}{\beta} \leq \frac{2}{\beta}\sum_{i=0}^{N-1}2^{(2/p-1)i}
  \leq \frac{2^{(2/p-1)(N-1)+1}}{(1-2^{1-2/p})\beta}
  \leq \frac{2\rho(f(x_0)-f^*)^{2/p-1}}{(1-2^{1-2/p})\beta\mu^{2/p}}
\end{align*}
and similarly, the number of null steps needed by~\eqref{eq:LipschitzGrowth-null-case4} as
\begin{align*}
  &\sum_{i=0}^{N-1} \frac{2^{(4/p-3)i+5}M^2}{\beta(1-\beta)^2\rho(\rho/\mu^{2/p})^{1/(1-2/p)}}\\
  &\leq \frac{32 M^2}{\beta(1-\beta)^2\rho(\rho/\mu^{2/p})^{1/(1-2/p)}}\sum_{i=0}^{N-1}2^{(4/p-3)i}\\
  &\leq \frac{32 M^2}{\beta(1-\beta)^2\rho(\rho/\mu^{2/p})^{1/(1-2/p)}}\\
  & \hspace{2cm}\max\left\{\frac{(f(x_0)-f^*)^{4/p-3}}{(\rho/\mu^{2/p})^{(4/p-3)/(1-2/p)}}, 1\right\} \min\left\{\frac{1}{1-2^{-|4/p-3|}}, N\right\}
\end{align*}
where the last inequality bounds the sum regardless of the sign of the exponent $4/p-3$.

\subsubsection{Given $1\leq p<2$, bounding steps with  $(\rho/\mu^{2/p})^{1/(1-2/p)}\geq \delta_k >\epsilon$.}
Finally, we show that the number of descent steps with $(\rho/\mu^{2/p})^{1/(1-2/p)}\geq \delta_k >\epsilon$ is bounded by
\begin{equation}\label{eq:LipschitzGrowth-descent-case5}
  \left\lceil \frac{\log\left(\frac{(\rho/\mu^{2/p})^{1/(1-2/p)}}{\epsilon}\right)}{-\log(1-\beta/2)} \right\rceil
\end{equation}
and the number of null steps with $(\rho/\mu^{2/p})^{1/(1-2/p)}\geq \delta_k >\epsilon$ is at most
\begin{equation}\label{eq:LipschitzGrowth-null-case5}
  \frac{16M^2}{\beta(1-\beta)^2\rho\epsilon} \ .
\end{equation}

In this case, our recurrence relation simplifies to have geometric decrease at each descent step $\delta_{k+1}\leq (1-\beta/2)\delta_k$. This immediately bounds the number of descent steps by~\eqref{eq:LipschitzGrowth-descent-case5}. Index the descent steps after a $(\rho/\mu^{2/p})^{1/(1-2/p)}$-minimizer but before an $\epsilon$-minimizer is found by $k_1<\dots< k_n$ such that $x_{k_n+1}$ is the first iterate with objective value less than $\epsilon$. Then for each $i=0\dots n-1$,
$ f(x_{k_i+1}) - f^* \geq (1-\beta/2)^{i-(n-1)}\epsilon.$
It follows from~\eqref{eq:prox-gap-bound} that
$\Delta_{k_i+1} \geq (f(x_{k_i+1})-f^*)/2 \geq \left(1-\beta/2\right)^{i-(n-1)}\epsilon/2.$
Plugging this into Lemma~\ref{lem:null-prox-gap} (note $\Delta_{k+T}=\Delta_{k+1}$ since the stepsize is held constant) upper bounds the number of consecutive null steps after the descent step $k_i$ by
$$ k_{i+1}-k_{i}-1 \leq \left(1-\beta/2\right)^{(n-1)-i}\frac{8M^2}{(1-\beta)^2\rho\epsilon} \ .$$
Summing this over $i=0\dots n-1$ bounds the additional number of null steps before an $\epsilon$-minimizer is found by~\eqref{eq:LipschitzGrowth-null-case5} as
$$ \sum_{i=0}^{n-1} \left(1-\beta/2\right)^{(n-1)-i}\frac{8M^2}{(1-\beta)^2\rho\epsilon} \leq \frac{16M^2}{\beta(1-\beta)^2\rho\epsilon} \ .$$

\subsection{Proof of Theorem~\ref{thm:SmoothGrowth}}\label{proof:SmoothGrowth}
Our bound on the number of descent steps comes directly from Theorem~\ref{thm:LipschitzGrowth}.
Our claimed bound on the total number of null steps follows by multiplying this by the constant bound on the number of consecutive null steps from Lemma~\ref{lem:null-prox-gap}.

\if\siam1
\else
\subsection{Proof of Theorem~\ref{thm:OPT-LipschitzGeneral}} \label{proof:OPT-LipschitzGeneral}
Combining the lower bound $\Delta_k \geq \frac{1}{2}(f(x_k)-f^*)$ with Lemma~\ref{lem:descent-prox-gap} shows linear decrease in the objective every descent step
$$ f(x_{k+1}) - f^* \leq \left(1 - \frac{\beta}{2}\right)(f(x_{k})-f^*).$$
Our bound on the number of descent steps follows immediately from this.
Combining the lower bound $\Delta_k \geq \frac{1}{2}(f(x_k)-f^*)$ with Lemma~\ref{lem:null-prox-gap} shows that at most
$$\frac{8M^2D^2}{(1-\beta)^2(f(x_{k+1})-f^*)^2}$$
null steps occur between each descent step.
Denote the sequence of descent steps taken by the bundle method by $k_1,k_2,k_3\dots$ and as a base case define $k_{0}=-1$. Let $k_n$ be the first descent step finding an $\epsilon$-minimizer, which must have $n\leq \lceil\log_{(1-\beta/2)}(\frac{\epsilon}{f(x_0)-f^*})\rceil_+$.
From our linear decrease condition, we know for any $i=0,1,2,3,\dots n-1$
$$ f(x_{k_i+1}) - f^* \geq \left(1 - \beta/2\right)^{i-(n-1)}\epsilon$$
and from our null step bound, we know for any $i=0,1,2,\dots n-1 $
$$ k_{i+1}-k_{i}-1 \leq \frac{8M^2D^2}{(1-\beta)^2(f(x_{k_i+1})-f^*)^2} \leq (1-\beta/2)^{2(i-(n-1))}\frac{8M^2D^2}{(1-\beta)^2\epsilon^2} \ .$$
Then summing up our null step bounds ensures
$$ k_n - n \leq \sum^n_{i=1} (1-\beta/2)^{2(i-1-(n-1))}\frac{8M^2D^2}{(1-\beta)^2\epsilon^2}.$$
Bounding this geometric series shows us that the bundle method finds an $\epsilon$-minimizer with the number of null steps bounded by
$$ \left(\frac{1}{1-(1-\beta/2)^2}\right)\frac{8M^2D^2}{(1-\beta)^2\epsilon^2} \ . $$

\subsection{Proof of Theorem~\ref{thm:OPT-LipschitzGrowth}}\label{proof:OPT-LipschitzGrowth}
Our bound on the number of descent steps follows from Theorem~\ref{thm:OPT-LipschitzGeneral}. Our proof of the null step bound follows the same approach as Theorem~\ref{thm:OPT-LipschitzGeneral} with only minor differences.
Applying Lemma~\ref{lem:null-prox-gap} with our stepsize choice~\eqref{eq:OPTholder-stepsize} bounds the number of consecutive null steps after some descent step $k$ by
$$\frac{8M^2}{(1-\beta)^2\mu^{2/p}(f(x_{k+1})-f^*)^{2-2/p}} \ .$$
Denote the descent steps $-1=k_0 < k_1 < k_2 < \dots$ and suppose the $x_{k_n+1}$ is the first $\epsilon$-minimizer. Then
$$ k_{i+1}-k_{i}-1 \leq (1-\beta/2)^{(2-2/p)(i-(n-1))}\frac{8M^2}{(1-\beta)^2\mu^{2/p}\epsilon^{2-2/p}}$$
since $f(x_{k_i+1}) - f^* \geq \left(1 - \frac{\beta}{2}\right)^{i-(n-1)}\epsilon$.
Summing this up gives
$$ k_n -n \leq \sum^n_{i=1} (1-\beta/2)^{(2-2/p)(i-1-(n-1))}\frac{8M^2}{(1-\beta)^2\mu^{2/p}\epsilon^{2-2/p}} \ .$$

When $p>1$, this geometric series shows us that the bundle method finds an $\epsilon$-minimizer with the number of null steps bounded by
$$ \left(\frac{1}{1-(1-\beta/2)^{2-2/p}}\right)\frac{8M^2}{(1-\beta)^2\mu^{2/p}\epsilon^{2-2/p}}.$$

When $p=1$, we have a constant upper bound on the number of null steps following a descent step. Hence the number of null steps is bounded by
$$\frac{8M^2}{(1-\beta)^2\mu^2}\left\lceil\frac{\log\left(\frac{f(x_0)-f^*}{\epsilon}\right)}{-\log(1-\beta/2)}\right\rceil \ . $$

\fi

\subsection{Proof of Theorem~\ref{thm:OPT-ParallelRate}} \label{proof:parallel}

Let $\delta_k = \min_{j\in\{0,\dots, J-1\}}\{f(x^{(j)}_{k})-f^*\}$ denote the lowest objective gap among all of our $J$ instances of the bundle method after they have taken $k$ synchronous steps. Then the core of our convergence proof is bounding the number of iterations where this lowest objective gap is in the interval
$$ (1-\beta/2)^{-n}\epsilon \leq \delta_k \leq (1-\beta/2)^{-(n+1)}\epsilon \ . $$
for any integer $0\leq n < N:=\left\lceil \frac{\log((f(x_0)-f^*)/\epsilon)}{-\log(1-\beta/2)}\right\rceil$.
Within this interval, we focus on the instance
$$ j = \left\lceil \log_2\left(\frac{\mu^{2/p}((1-\beta/2)^{-n}\epsilon)^{1-2/p}}{4\bar{\rho}}\right)\right\rceil.$$
This instance of the bundle method's constant stepsize $\rho^{(j)}=2^j\bar{\rho}$ approximates the stepsize~\eqref{eq:OPTholder-stepsize} as
$$ \frac{1}{4}\mu^{2/p}((1-\beta/2)^{-n}\epsilon)^{1-2/p} \leq \rho^{(j)} \leq \frac{1}{2}\mu^{2/p}((1-\beta/2)^{-n}\epsilon)^{1-2/p} \ . $$
Then~\eqref{eq:prox-gap-bound-growth} bounds this method's proximal gap before an $(1-\beta/2)^{-n}\epsilon$-minimizer is found by
$$ \Delta^{(j)}_k \geq \frac{1}{2}(f(x^{(j)}_k)-f^*) \geq (1-\beta/2)^{-n}\epsilon/2 \ .$$

Letting $\delta^{(j)}_k = f(x^{(j)}_k)-f^*$, each descent step $k$ improves method $j$'s objective gap according to the recurrence
$ \delta^{(j)}_{k+1} \leq \min\{(1-\beta/2)\delta^{(j)}_k, \delta_k \} $
where the first term in the minimum comes from Lemma~\ref{lem:descent-prox-gap} and the second term comes from method $j$ taking any further improvement from the other bundle methods. By assumption, we have $\delta_k\leq (1-\beta/2)^{-(n+1)}\epsilon$, and so after one descent step $k'>k$ we must have
$ \delta^{(j)}_{k'+1} \leq (1-\beta/2)^{-(n+1)}\epsilon $. Thus after a second descent step $k''>k'$, our intermediate target accuracy is met as $\delta_{k''+1} \leq \delta^{(j)}_{k''+1} \leq (1-\beta/2)^{-n}\epsilon$.

Applying Lemma~\ref{lem:null-prox-gap} bounds the number of null steps between descent steps by
$$ \frac{8M^2}{(1-\beta)^2\rho^{(j)}\Delta^{(j)}_{k+1}} \leq \frac{64M^2}{(1-\beta)^2\mu^{2/p}((1-\beta/2)^{-n}\epsilon)^{2-2/p}} \ .$$
Hence the total number of steps before $\delta^{(j)}_k < 2^n\epsilon$ (and consequently $\delta_k < 2^n\epsilon$) is at most
$$ 2\left(\frac{64M^2}{(1-\beta)^2\mu^{2/p}((1-\beta/2)^{-n}\epsilon)^{2-2/p}} +1\right)\ . $$
Summing over this bound completes our proof. When $p>1$, this gives
\begin{align*}
  &\sum_{n=0}^{N-1} 2\left(\frac{64M^2}{(1-\beta)^2\mu^{2/p}((1-\beta/2)^{-n}\epsilon)^{2-2/p}} +1\right)\\
  &= 2\sum_{n=0}^{N-1}\frac{64M^2}{(1-\beta)^2\mu^{2/p}((1-\beta/2)^{-n}\epsilon)^{2-2/p}} + 2\left\lceil \frac{\log((f(x_0)-f^*)/\epsilon)}{-\log(1-\beta/2)}\right\rceil\\
  &\leq \left(\frac{2}{1-(1-\beta/2)^{2-2/p}}\right)\frac{64M^2}{(1-\beta)^2\mu^{2/p}\epsilon^{2-2/p}} + 2\left\lceil \frac{\log((f(x_0)-f^*)/\epsilon)}{-\log(1-\beta/2)}\right\rceil \ .
\end{align*}

When $p=1$, the number of steps in each of our intervals is constant. Consequently, the total number of iterations before an $\epsilon$ minimizer is found is at most
$$ \sum_{n=0}^{N-1} 2\left(\frac{64M^2}{(1-\beta)^2\mu^{2}} +1\right) = 2\left(\frac{64M^2}{(1-\beta)^2\mu^{2}} +1\right)\left\lceil \frac{\log((f(x_0)-f^*)/\epsilon)}{-\log(1-\beta/2)}\right\rceil\ . $$


\section*{Acknowledgements} We would like to thank Haihao Lu for pointing out a number of typos in an early version of this manuscript. We would also like to thank the anonymous reviewers for their insightful comments and feedback. Finally, we thank Adrian Lewis for introducing us to bundle methods.

{\small
\bibliographystyle{siamplain}
\bibliography{bibliography}

\def\cfac#1{\ifmmode\setbox7\hbox{$\accent"5E#1$}\else
  \setbox7\hbox{\accent"5E#1}\penalty 10000\relax\fi\raise 1\ht7
  \hbox{\lower1.15ex\hbox to 1\wd7{\hss\accent"13\hss}}\penalty 10000
  \hskip-1\wd7\penalty 10000\box7}
\begin{thebibliography}{10}

\bibitem{libsvm}
{\em Libsvm data: Classification (binary class)}.
\newblock
  \url{https://www.csie.ntu.edu.tw/~cjlin/libsvmtools/datasets/binary.html}.
\newblock Accessed: 2021-05-12.

\bibitem{Apkarian2009}
{\sc P.~Apkarian, D.~Noll, and O.~Prot}, {\em {A Proximity Control Algorithm to
  Minimize Nonsmooth and Nonconvex Semi-infinite Maximum Eigenvalue
  Functions}}, J. Convex Anal., 16 (2009), pp.~641--666.

\bibitem{bonnans1995family}
{\sc J.~F. Bonnans, J.~C. Gilbert, C.~Lemar{\'e}chal, and C.~A.
  Sagastiz{\'a}bal}, {\em A family of variable metric proximal methods},
  Mathematical Programming, 68 (1995), pp.~15--47.

\bibitem{Burke1993}
{\sc J.~V. Burke and M.~C. Ferris}, {\em Weak sharp minima in mathematical
  programming}, SIAM Journal on Control and Optimization, 31 (1993),
  pp.~1340--1359.

\bibitem{byrd1994representations}
{\sc R.~H. Byrd, J.~Nocedal, and R.~B. Schnabel}, {\em Representations of
  quasi-newton matrices and their use in limited memory methods}, Mathematical
  Programming, 63 (1994), pp.~129--156.

\bibitem{charisopoulos2021low}
{\sc V.~Charisopoulos, Y.~Chen, D.~Davis, M.~D{\'\i}az, L.~Ding, and
  D.~Drusvyatskiy}, {\em Low-rank matrix recovery with composite optimization:
  good conditioning and rapid convergence}, Foundations of Computational
  Mathematics,  (2021), pp.~1--89.

\bibitem{correa1993convergence}
{\sc R.~Correa and C.~Lemar{\'e}chal}, {\em Convergence of some algorithms for
  convex minimization}, Mathematical Programming, 62 (1993), pp.~261--275.

\bibitem{Davis_2019}
{\sc D.~Davis and D.~Drusvyatskiy}, {\em Stochastic model-based minimization of
  weakly convex functions}, SIAM Journal on Optimization, 29 (2019),
  p.~207–239.

\bibitem{de2017target}
{\sc W.~de~Oliveira}, {\em Target radius methods for nonsmooth convex
  optimization}, Operations Research Letters, 45 (2017), pp.~659--664.

\bibitem{deOliveira2019}
{\sc W.~de~Oliveira}, {\em Proximal bundle methods for nonsmooth {DC}
  programming}, J. Glob. Optim., 75 (2019), pp.~523--563.

\bibitem{deOliveira2014}
{\sc W.~de~Oliveira, C.~A. Sagastiz{\'{a}}bal, and C.~Lemar{\'{e}}chal}, {\em
  Convex proximal bundle methods in depth: a unified analysis for inexact
  oracles}, Math. Program., 148 (2014), pp.~241--277.

\bibitem{de2016doubly}
{\sc W.~De~Oliveira and M.~Solodov}, {\em A doubly stabilized bundle method for
  nonsmooth convex optimization}, Mathematical programming, 156 (2016),
  pp.~125--159.

\bibitem{deOliveira2020}
{\sc W.~de~Oliveira and M.~Solodov}, {\em Bundle Methods for Inexact Data},
  Springer International Publishing, Cham, 2020, pp.~417--459.

\bibitem{Ding2020}
{\sc L.~Ding and B.~Grimmer}, {\em Revisit of spectral bundle methods:
  Primal-dual (sub)linear convergence rates}, 2020,
  \url{https://arxiv.org/abs/2008.07067}.

\bibitem{Drusvyatskiy2021}
{\sc D.~Drusvyatskiy, A.~D. Ioffe, and A.~S. Lewis}, {\em Nonsmooth
  optimization using taylor-like models: error bounds, convergence, and
  termination criteria}, Math. Program., 185 (2021), pp.~357--383.

\bibitem{Du2017}
{\sc Y.~Du and A.~Ruszczy\'{n}ski}, {\em {Rate of Convergence of the Bundle
  Method}}, J. Optim. Theory Appl., 173 (2017), pp.~908--922.

\bibitem{fischer2014parallel}
{\sc F.~Fischer and C.~Helmberg}, {\em A parallel bundle framework for
  asynchronous subspace optimization of nonsmooth convex functions}, SIAM
  Journal on Optimization, 24 (2014), pp.~795--822.

\bibitem{Frangioni2020}
{\sc A.~Frangioni}, {\em Standard Bundle Methods: Untrusted Models and
  Duality}, Springer International Publishing, Cham, 2020, pp.~61--116.

\bibitem{Warren2010}
{\sc W.~Hare and C.~Sagastiz\'{a}bal}, {\em {A Redistributed Proximal Bundle
  Method for Nonconvex Optimization}}, SIAM J. Optim., 20 (2010),
  pp.~2442--2473.

\bibitem{Hare2016}
{\sc W.~Hare, C.~Sagastiz{\'a}bal, and M.~Solodov}, {\em {A Proximal Bundle
  Method for Nonsmooth Nonconvex Functions with Inexact Information}},
  Computational Optimization and Applications, 63 (2016), pp.~1--28.

\bibitem{Helmberg2000}
{\sc C.~Helmberg and F.~Rendl}, {\em A spectral bundle method for semidefinite
  programming}, SIAM Journal on Optimization, 10 (2000), pp.~673--696.

\bibitem{hiriart2013convex}
{\sc J.-B. Hiriart-Urruty and C.~Lemar{\'e}chal}, {\em Convex Analysis and
  Minimization Algorithms II: Advanced Theory and Bundle Methods}, vol.~306,
  Springer Berlin Heidelberg, 1993.

\bibitem{iutzeler2020asynchronous}
{\sc F.~Iutzeler, J.~Malick, and W.~de~Oliveira}, {\em Asynchronous level
  bundle methods}, Mathematical Programming, 184 (2020), pp.~319--348.

\bibitem{kim2019asynchronous}
{\sc K.~Kim, C.~G. Petra, and V.~M. Zavala}, {\em An asynchronous
  bundle-trust-region method for dual decomposition of stochastic mixed-integer
  programming}, SIAM Journal on Optimization, 29 (2019), pp.~318--342.

\bibitem{Kiwiel1983}
{\sc K.~C. Kiwiel}, {\em {An Aggregate Subgradient Method for Nonsmooth Convex
  Minimization}}, Math. Program., 27 (1983), pp.~320--341.

\bibitem{Kiwiel1985-nonconvex}
{\sc K.~C. Kiwiel}, {\em {A Linearization Algorithm for Nonsmooth
  Minimization}}, Mathematics of Operations Research, 10 (1985), pp.~185--194.

\bibitem{Kiwiel1985}
{\sc K.~C. {Kiwiel}}, {\em {Methods of Descent for Nondifferentiable
  Optimization.}}, Springer, Berlin, 1985.

\bibitem{kiwiel1990proximity}
{\sc K.~C. Kiwiel}, {\em Proximity control in bundle methods for convex
  nondifferentiable minimization}, Mathematical programming, 46 (1990),
  pp.~105--122.

\bibitem{Kiwiel1995}
{\sc K.~C. Kiwiel}, {\em {Proximal Level Bundle Methods for Convex
  Nondifferentiable Optimization, Saddle-point Problems and Variational
  Inequalities}}, Math. Program., 69 (1995), pp.~89--109.

\bibitem{Kiwiel2000}
{\sc K.~C. Kiwiel}, {\em {Efficiency of Proximal Bundle Methods}}, Journal of
  Optimization Theory and Applications, 104 (2000), pp.~589--603.

\bibitem{kiwiel2006proximal}
{\sc K.~C. Kiwiel}, {\em A proximal bundle method with approximate subgradient
  linearizations}, SIAM Journal on optimization, 16 (2006), pp.~1007--1023.

\bibitem{Lan2015}
{\sc G.~Lan}, {\em {Bundle-Level Type Methods Uniformly Optimal For Smooth And
  Nonsmooth Convex Optimization}}, Mathematical Programming, 149 (2015),
  pp.~1--45.

\bibitem{Lemarechal1975}
{\sc C.~Lemarechal}, {\em {An Extension of Davidon Methods to Nondifferentiable
  Problems}}, Springer Berlin Heidelberg, Berlin, Heidelberg, 1975,
  pp.~95--109.

\bibitem{Lemarechal2001}
{\sc C.~Lemar{\'e}chal}, {\em Lagrangian Relaxation}, Springer Berlin
  Heidelberg, Berlin, Heidelberg, 2001, pp.~112--156.

\bibitem{Lemarechal1995}
{\sc C.~Lemar{\'e}chal, A.~Nemirovskii, and Y.~Nesterov}, {\em {New Variants of
  Bundle Methods}}, Math. Program., 69 (1995), pp.~111--147.

\bibitem{lemarechal1997variable}
{\sc C.~Lemar{\'e}chal and C.~Sagastiz{\'a}bal}, {\em Variable metric bundle
  methods: from conceptual to implementable forms}, Mathematical Programming,
  76 (1997), pp.~393--410.

\bibitem{Liang2021}
{\sc J.~Liang and R.~D.~C. Monteiro}, {\em A proximal bundle variant with
  optimal iteration-complexity for a large range of prox stepsizes}, 2021,
  \url{https://arxiv.org/abs/2003.11457}.

\bibitem{Lv2018}
{\sc J.~Lv, L.~Pang, and F.~Meng}, {\em A proximal bundle method for
  constrained nonsmooth nonconvex optimization with inexact information}, J.
  Glob. Optim., 70 (2018), pp.~517--549.

\bibitem{Lv2019}
{\sc J.~Lv, L.~Pang, N.~Xu, and Z.~Xiao}, {\em An infeasible bundle method for
  nonconvex constrained optimization with application to semi-infinite
  programming problems}, Numer. Algorithms, 80 (2019), pp.~397--427.

\bibitem{mifflin1977algorithm}
{\sc R.~Mifflin}, {\em An algorithm for constrained optimization with
  semismooth functions}, Mathematics of Operations Research, 2 (1977),
  pp.~191--207.

\bibitem{Mifflin1982}
{\sc R.~Mifflin}, {\em {A Modification and an Extension of Lemarechal's
  Algorithm for Nonsmooth Minimization}}, Springer Berlin Heidelberg, Berlin,
  Heidelberg, 1982, pp.~77--90.

\bibitem{mifflin2005algorithm}
{\sc R.~Mifflin and C.~Sagastiz{\'a}bal}, {\em A vu-algorithm for convex
  minimization}, Mathematical programming, 104 (2005), pp.~583--608.

\bibitem{mifflin2012science}
{\sc R.~Mifflin and C.~Sagastiz{\'a}bal}, {\em A science fiction story in
  nonsmooth optimization originating at iiasa}, this volume,  (2012).

\bibitem{mogensen2018optim}
{\sc P.~K. Mogensen and A.~N. Riseth}, {\em Optim: A mathematical optimization
  package for {Julia}}, Journal of Open Source Software, 3 (2018), p.~615.

\bibitem{Monjezi2019b}
{\sc N.~H. Monjezi and S.~Nobakhtian}, {\em A new infeasible proximal bundle
  algorithm for nonsmooth nonconvex constrained optimization}, Comput. Optim.
  Appl., 74 (2019), pp.~443--480.

\bibitem{Monjezi2019a}
{\sc N.~H. Monjezi and S.~Nobakhtian}, {\em A filter proximal bundle method for
  nonsmooth nonconvex constrained optimization}, J. Glob. Optim., 79 (2021),
  pp.~1--37.

\bibitem{intro_lect}
{\sc Y.~Nesterov}, {\em Introductory lectures on convex optimization}, vol.~87
  of Applied Optimization, Kluwer Academic Publishers, Boston, MA, 2004.
\newblock A basic course.

\bibitem{Nesterov2105}
{\sc Y.~Nesterov and M.~I. Florea}, {\em Gradient methods with memory},
  Optimization Methods and Software, 0 (2021), pp.~1--18.

\bibitem{NW}
{\sc J.~Nocedal and S.~Wright}, {\em Numerical optimization}, Springer Series
  in Operations Research and Financial Engineering, Springer, New York,
  second~ed., 2006.

\bibitem{Ochs2019}
{\sc P.~Ochs, J.~Fadili, and T.~Brox}, {\em Non-smooth non-convex bregman
  minimization: Unification and new algorithms}, J. Optim. Theory Appl., 181
  (2019), pp.~244--278.

\bibitem{Oustry2000}
{\sc F.~Oustry}, {\em A second-order bundle method to minimize the maximum
  eigenvalue function}, Math. Program., 89 (2000), pp.~1--33.

\bibitem{polyak1987introduction}
{\sc B.~T. Polyak}, {\em Introduction to optimization. optimization software},
  Inc., Publications Division, New York, 1 (1987), p.~32.

\bibitem{RenegarGrimmer2021}
{\sc J.~{Renegar} and B.~{Grimmer}}, {\em {A Simple Nearly-Optimal Restart
  Scheme For Speeding-Up First Order Methods}}, Foundations of Computational
  Mathematics,  (2021).

\bibitem{Ruszczynski2006}
{\sc A.~Ruszczynski}, {\em {Nonlinear Optimization}}, Princeton University
  Press, Princeton, NJ, USA, 2006.

\bibitem{Sagastizabal2012}
{\sc C.~Sagastiz{\'a}bal}, {\em {Divide to Conquer: Decomposition Methods for
  Energy Optimization}}, Mathematical Programming, 134 (2012), pp.~187--222.

\bibitem{Sagastizabal2005}
{\sc C.~Sagastiz\'{a}bal and M.~Solodov}, {\em {An Infeasible Bundle Method for
  Nonsmooth Convex Constrained Optimization without a Penalty Function or a
  Filter}}, SIAM Journal on Optimization, 16 (2005), pp.~146--169.

\bibitem{shalev2011pegasos}
{\sc S.~Shalev-Shwartz, Y.~Singer, N.~Srebro, and A.~Cotter}, {\em Pegasos:
  Primal estimated sub-gradient solver for svm}, Mathematical programming, 127
  (2011), pp.~3--30.

\bibitem{solodov2003approximations}
{\sc M.~V. Solodov}, {\em On approximations with finite precision in bundle
  methods for nonsmooth optimization}, Journal of Optimization Theory and
  Applications, 119 (2003), pp.~151--165.

\bibitem{Wolfe1975}
{\sc P.~Wolfe}, {\em {A Method of Conjugate Subgradients for Minimizing
  Nondifferentiable Functions}}, Springer Berlin Heidelberg, Berlin,
  Heidelberg, 1975, pp.~145--173.

\end{thebibliography}
}
\appendix
\section{Solutions to Recurrence Relations}

Throughout our analysis, we frequently encounter recurrence relations of the form
$ \delta_{k+1} \leq \delta_{k} - \alpha \delta_k^q $ for some $\alpha>0$ and $q>1$.
The following lemma bounds the number of steps of such a recurrence to reach a desired level of accuracy $\delta_k\leq \epsilon$.
\begin{lemma}\label{lem:recurrence}
	Assume that a sequence $\{\delta_{k}\}_{k=1}^{\infty},$ satisfies the recurrence $\delta_{k+1} \leq \delta_{k} - \alpha \delta_k^q$. Then for any $\epsilon > 0$, the inequality $\delta_k\leq \epsilon$ holds for some
	\begin{equation*}
		k\leq \left\lceil\dfrac{1}{(q-1)\alpha\epsilon^{q-1}}\right\rceil \ .
	\end{equation*}
\end{lemma}
\begin{proof}
It suffices to show the following upper bound on $\delta_k$ as a function of $k$
$$ \delta_k \leq \left(\frac{1}{(q-1)\alpha k}\right)^{1/(q-1)} \ .$$
First we show this bound holds with $k=1$. This follows as
\begin{align*}
	\delta_1 \leq \delta_0-\alpha\delta_0^q \leq \max_{\delta\in\RR}\{\delta-\alpha\delta^q\} \leq \left(\frac{1}{q\alpha}\right)^{1/(q-1)} \ .
\end{align*}
Then we complete our proof by induction using the following {\it weighted arithmetic-geometric mean (AM-GM) inequality}, which ensures for any $a,\alpha,b,\beta>0$ we have
$ a^\alpha b^\beta\leq \left(\frac{\alpha a + \beta b}{\alpha+\beta}\right)^{\alpha+\beta} \ .$
This implies that for any $k\geq 1$,  $(k-(q-1)^{-1})(k+1)^{1/(q-1)} \leq k^{q/(q-1)}$ by taking $a=k-(q-1)^{-1}$, $\alpha=1$, $b=k+1$, $\beta=1/(q-1)$. By expanding the recurrence at $k+1$ and applying this inequality we get
\begin{align*}
	\delta_{k+1} \leq \delta_k - \alpha\delta_k^q &\leq \left(\frac{1}{(q-1)\alpha k}\right)^{1/(q-1)} -\alpha\left(\frac{1}{(q-1)\alpha k}\right)^{q/(q-1)} \\
	&= \left(\frac{1}{(q-1)\alpha}\right)^{1/(q-1)} \left(\frac{k}{k^{q/(q-1)}} - \frac{1}{(q-1)k^{q/(q-1)}}\right)\\
	&= \left(\frac{1}{(q-1)\alpha}\right)^{1/(q-1)}\frac{k-(q-1)^{-1}}{k^{q/(q-1)}}\\
                                                      &\leq \left(\frac{1}{(q-1)\alpha(k+1)}\right)^{1/(q-1)} \ .
\end{align*}
Proving the result.
\end{proof}

\section{Proof of Claim~\ref{claim:closedForm}} \label{app:proofClaim} For notational convenience define $h_{1}(\cdot) := f_{t}(z_{t+1}) + \langle s_{t+1}, \cdot-z_{t+1}\rangle$ and $h_{2}(\cdot) :=  f(z_{t+1}) + \langle g_{t+1}, \cdot-z_{t+1}\rangle$ and so $\tilde f_{t+1} = \max\{h_{1}, h_{2}\}.$ Since the problem is strongly convex, there is unique solution. The solution is characterized by
		\begin{equation} \label{eq:opt}  y_{t+2} \in x_{k+1} - \frac{1}{\rho_{t}}\partial \tilde f_{t+1}(y_{t+2})
		\end{equation}
		where subdifferential is
		$$\partial \tilde f_{t+1}(y_{t+2}) =
		\begin{cases} \{s_{t+1}\} & \text{if } h_{1}(y_{t+2}) >h_{2}(y_{t+2}),\\
		\{g_{t+1}\} & \text{if }  h_{2}(y_{t+2}) > h_{1}(y_{t+2}),\\
		\conv(\{s_{t+1}, g_{t+1}\})& \text{otherwise,}
		\end{cases}$$
		and $\conv$ denotes the convex hull.

		Define $w = \theta_{t+1}g_{t+1} +(1-\theta_{t+1})s_{t+1}$.  In light of \eqref{eq:opt}, the claim would follow if we had $w \in \partial \tilde f_{t+1}(y_{t+2})$. Let's prove that this is the case. Observe that the coefficient $\theta_{t+1} \in [0, 1]$ because $f \geq f_{t}.$ We consider two cases.

		Case 1. Assume $\theta_{t+1}  = 1 < \frac{\rho_{t}\left(f(z_{t+1})-f_t(z_{t+1})\right)}{\|g_{t+1}-s_{t+1}\|^2},$ then $w = g_{t+1}$. Moreover,
		\begin{align*}  h_{1}(y_{k+2})-h_{2}(y_{k+2}) &= f_{t}(z_{t+1}) - f(z_{t+1}) + \langle s_{t+1} - g_{t+1}, y_{t+2}-z_{t+1}\rangle\\
		& < -\frac{1}{\rho_{t}}\|g_{t+1}-  s_{t+1}\|^{2} + \langle s_{t+1} - g_{t+1}, x_{k+1} -z_{t+1}  - \frac{1}{\rho_{t}} g_{t+1}\rangle \\
		& = -\frac{1}{\rho_{t}}\|g_{t+1}-  s_{t+1}\|^{2} + \frac{1}{\rho_{t}}\|g_{t+1}-  s_{t+1}\|^{2} = 0 ,
		\end{align*}
		where we used  $s_{t+1} = \rho_{t}(x_{k+1} - z_{t+1}).$ Thus $w \in \partial \tilde f_{t+1}(y_{t+2}).$

		Case 2. Assume $\theta_{t+1}  = \frac{\rho_{t}\left(f(z_{t+1})-f_t(z_{t+1})\right)}{\|g_{t+1}-s_{t+1}\|^2} \leq 1$. Then,
		\begin{align*}
		h_{1}(y_{k+2})-h_{2}(y_{k+2}) &= f_{t}(z_{t+1}) - f(z_{t+1}) + \left\langle s_{t+1} - g_{t+1}, x_{k+1} -z_{t+1}  - \frac{1}{\rho_{t}} w\right\rangle \\
		& =  f_{t}(z_{t+1}) - f(z_{t+1}) + \theta_{t+1} \| s_{t+1} - g_{t+1}\|^{2} = 0
		\end{align*}
		where the middle equality follows from the definition of $s_{t+1}$ and the last equality uses the definition of $\theta_{t+1}.$ We conclude that $w \in \partial \tilde f_{t+1}(y_{t+2}).$

\end{document}